\documentclass[a4paper, 12pt]{article}

\usepackage{amsmath,amssymb,amsthm}
\usepackage{setspace}
\allowdisplaybreaks

\newtheorem{theorem}{Theorem}[section]
\newtheorem{proposition}[theorem]{Proposition}
\newtheorem{definition}[theorem]{Definition}
\newtheorem{lemma}[theorem]{Lemma}

\theoremstyle{remark}
\newtheorem{remark}[theorem]{Remark}
\newtheorem{example}[theorem]{Example}

\numberwithin{equation}{section}

\DeclareFontEncoding{OT2}{}{}
\DeclareFontSubstitution{OT2}{cmr}{m}{n}
\DeclareFontFamily{OT2}{cmr}{\hyphenchar\font45 }
\DeclareFontShape{OT2}{cmr}{m}{n}{<5><6><7><8><9>gen*wncyr<10><10.95><12><14.4><17.28><20.74><24.88>wncyr10}{}
\DeclareFontShape{OT2}{cmr}{b}{n}{<5><6><7><8><9>gen*wncyb<10><10.95><12><14.4><17.28><20.74><24.88>wncyb10}{}
\DeclareMathAlphabet{\mathcyr}{OT2}{cmr}{m}{n}
\DeclareMathAlphabet{\mathcyb}{OT2}{cmr}{b}{n}
\SetMathAlphabet{\mathcyr}{bold}{OT2}{cmr}{b}{n}
\newcommand{\sh}{\mathcyr {sh}}

\newcommand{\h}{\mathfrak H}
\newcommand{\R}{\mathbb R}
\newcommand{\Z}{{\mathbb Z}}
\newcommand{\N}{{\mathbb N}}
\newcommand{\C}{{\mathbb C}}
\newcommand{\Q}{{\mathbb Q}}
\newcommand{\x}{{\bf x}}
\newcommand{\y}{{\bf y}}
\newcommand{\cG}{\widetilde{G}}

\title{The double shuffle relations\\ for multiple Eisenstein series}
\author{Henrik Bachmann\footnote{email : henrik.bachmann@uni-hamburg.de, Universit\"{a}t Hamburg}, Koji Tasaka\footnote{email : koji.tasaka@math.nagoya-u.ac.jp, Graduate School of Mathematics, Nagoya University}}
\date{}

\begin{document}

\maketitle

\begin{abstract}
We study the multiple Eisenstein series introduced by Gangl, Kaneko and Zagier.
We give a proof of (restricted) finite double shuffle relations for multiple Eisenstein series by revealing an explicit connection between the Fourier expansion of multiple Eisenstein series and the Goncharov coproduct on Hopf algebras of iterated integrals.
\end{abstract}
{\bf Subjclass[2010]:} 11M32, 11F11, 13J05, 33E20.


\section{Introduction}

The purpose of this paper is to study the multiple Eisenstein series, which are holomorphic functions on the upper half-plane $\{\tau\in\C\mid {\rm Im}(\tau)>0\}$ and which can be viewed as a multivariate generalization of the classical Eisenstein series, defined as an iterated multiple sum
\begin{equation}\label{eq1_1}
G_{n_1,\ldots,n_r}(\tau)= \sum_{\substack{0\prec \lambda_1\prec\cdots\prec \lambda_r\\ \lambda_1,\ldots,\lambda_r \in \Z\tau+\Z}} \frac{1}{\lambda_1^{n_1}\cdots \lambda_r^{n_r}}\quad (n_1,\ldots,n_{r-1}\in\Z_{\ge2},n_r\in\Z_{\ge3}),
\end{equation}
where the positivity $l\tau+m\succ 0$ of a lattice point is defined to be either $l>0$ or $l=0$, $m>0$, and $l\tau+m\succ l'\tau+m'$ means $(l-l')\tau+(m-m')\succ0$.
These functions were first introduced and studied by Gangl, Kaneko and Zagier \cite[Section 7]{GKZ}, where they investigated the double shuffle relation satisfied by double zeta values for the double Eisenstein series $G_{n_1,n_2}(\tau)$.
Here the double zeta value is the special case of multiple zeta values defined by
\begin{equation} \label{eq1_2}
 \zeta(n_1,\ldots,n_r)=\sum_{\substack{0<m_1<\cdots<m_r\\m_1,\ldots,m_r\in\Z}} \frac{1}{m_1^{n_1}\cdots m_r^{n_r}} \quad (n_1,\ldots,n_{r-1}\in\Z_{\ge1},n_r\in\Z_{\ge2}).
\end{equation}
Their results were extended to the double Eisenstein series for higher level (congruence subgroup of level $N$) in \cite{KT} ($N=2$) and in \cite{YZ} ($N$ : general), and have interesting applications to the theory of modular forms (see \cite{Tasaka}) as well as the study of double zeta values of level $N$.
Our aim of this paper is to give a framework of and a proof of double shuffle relations for multiple Eisenstein series.

The double shuffle relation, or rather, the finite double shuffle relation (cf. e.g. \cite{IKZ}) describes a collection of $\Q$-linear relations among multiple zeta values arising from two ways of expressing multiple zeta values as iterated sums \eqref{eq1_2} and as iterated integrals \eqref{eq3_1}.
Each expression produces an algebraic structure on the $\Q$-vector space spanned by all multiple zeta values.
The product associated to \eqref{eq1_2} (resp. \eqref{eq3_1}) is called the harmonic product (resp. shuffle product).
For example, using the harmonic product, we have
\begin{equation*}
\zeta(3)\zeta(3)=2\zeta(3,3)+\zeta(6),
\end{equation*}
and by the shuffle product formulas one obtains
\begin{equation}\label{eq1_3}
\zeta(3)\zeta(3)=12\zeta(1,5)+6\zeta(2,4)+2\zeta(3,3).
\end{equation}
Combining these equations gives the relation 
\[ 12\zeta(1,5)+6\zeta(2,4)-\zeta(6)=0.\]

For the multiple Eisenstein series \eqref{eq1_1}, it is easily seen that the harmonic product formulas hold when the series defining $G_{n_1,\ldots,n_r}(\tau)$ converges absolutely, i.e. $n_1,\ldots,n_{r-1}\in\Z_{\ge2}$ and $n_r\in\Z_{\ge3}$, but the shuffle product is not the case -- the shuffle product formula \eqref{eq1_3} replacing $\zeta$ with $G$ does not make sense because an undefined multiple Eisenstein series $G_{1,5}(\tau)$ is involved.
This paper develops the shuffle product of multiple Eisenstein series by revealing an explicit connection between the multiple Eisenstein series and the Goncharov coproduct, and as a consequence the validity of a restricted version of the finite double shuffle relations for multiple Eisenstein series is obtained.

This paper begins by computing the Fourier expansion of $G_{n_1,\ldots,n_r}(\tau)$ for $n_1,\ldots,n_r\ge2$ (the case $n_r=2$ will be treated by a certain limit argument in Definition \ref{2_1}) in Section 2.
The Fourier expansion is intimately related with the Goncharov coproduct $\Delta$ (see \eqref{eq3_4}) on Hopf algebras of iterated integrals introduced by Goncharov \cite[Section 2]{G}, which was first observed by Kaneko in several cases and studied by Belcher \cite{Bel}.
His Hopf algebra $\mathcal{I}_\bullet(S)$ is reviewed in Section 3.2, and we will observe a relationship between the Fourier expansion and the Goncharov coproduct $\Delta$ in the quotient Hopf algebra $\mathcal{I}_\bullet^1:=\mathcal{I}_\bullet/\mathbb{I}(0;0;1)\mathcal{I}_\bullet$ ($\mathcal{I}_\bullet:=\mathcal{I}_\bullet(\{0,1\})$, which can not be seen in $\mathcal{I}_\bullet$ itself.
The space $\mathcal{I}_{\bullet}^1$ has a linear basis (Proposition~\ref{3_5})
$$\{I(n_1,\ldots,n_r)\mid r\ge0,n_1,\ldots,n_r\in \Z_{>0}\},$$
and we will express the Goncharov coproduct $\Delta(I(n_1,\ldots,n_r))$ as a certain algebraic combination of the above basis (Propositions \ref{3_8} and \ref{3_11}).
As an example of this expression (see \eqref{eq3_10}), one has
\[  \Delta( I(2,3) ) =   I(2,3) \otimes 1  + 3 I(3) \otimes I(2) + 2 I(2) \otimes I(3) + 1 \otimes I(2,3) . \]
The relationship is then obtained by comparing the formula for $\Delta(I(n_1,\ldots,n_r))$ with the Fourier expansion of $G_{n_1,\ldots,n_r}(\tau)$, which in the case of $r=2$ can be found by \eqref{eq3_10} and \eqref{eq2_8}.
More precisely, let us define the $\Q$-linear maps $\mathfrak{z}^{\sh}:\mathcal{I}^1_\bullet \rightarrow \R$ and $\mathfrak{g}:\mathcal{I}^1_\bullet \rightarrow \C[[q]]$ given by $I(n_1,\ldots,n_r)\mapsto \zeta^{\sh}(n_1,\ldots,n_r)$ and $I(n_1,\ldots,n_r)\mapsto g_{n_1,\ldots,n_r}(q)$, where $\zeta^{\sh}(n_1,\ldots,n_r)$ is the regularized multiple zeta value with respect to the shuffle product (Definition \ref{3_1}) and $g_{n_1,\ldots,n_r}(q)$ is the generating series of the multiple divisor sum appearing in the Fourier expansion of multiple Eisenstein series (see \eqref{eq2_4}). 
For instance, by \eqref{eq2_8} we have
\[  G_{2,3}(\tau) =\zeta(2,3) + 3 \zeta(3) g_2(q)  + 2 \zeta(2) g_3(q) + g_{2,3}(q) ,\]
and hence $ \big( \mathfrak{z}^{\sh}\otimes \mathfrak{g} \big)\circ \Delta(I(2,3)) = G_{2,3}(\tau)$.
In general, we have the following theorem which is the first main result of this paper (proved in Section 3.5).
\begin{theorem} \label{1_1}
For integers $n_1,\ldots,n_r\ge2$ we have
\[ \big( \mathfrak{z}^{\sh}\otimes \mathfrak{g} \big) \circ \Delta(I(n_1,\ldots,n_r)) = G_{n_1,\ldots,n_r}(\tau) \quad (q=e^{2\pi \sqrt{-1}\tau}).\]
\end{theorem}

The maps $\Delta:\mathcal{I}^1_\bullet\rightarrow \mathcal{I}^1_\bullet \otimes \mathcal{I}^1_\bullet$ and $\mathfrak{z}^{\sh}:\mathcal{I}^1_\bullet\rightarrow \R$ are algebra homomorphisms (Propositions \ref{3_4} and \ref{3_7}) but the map $\mathfrak{g}:\mathcal{I}_\bullet^1 \rightarrow \C[[q]]$ is not an algebra homomorphism (see Remark \ref{4_7}).
Thus we can not expect a validity of the shuffle product formulas for the $q$-series $\big( \mathfrak{z}^{\sh}\otimes \mathfrak{g} \big) \circ \Delta(I(n_1,\ldots,n_r))$ ($n_1,\ldots,n_r\in\Z_{\ge1}$) which can be naturally regarded as an extension of $G_{n_1,\ldots,n_r}(\tau)$ to the indices with $n_i=1$. 

We shall construct in Section 4.1 an algebra homomorphism $\mathfrak{g}^\sh : \mathcal{I}_\bullet^1\rightarrow \C[[q]]$ (Definition \ref{4_5}) using certain $q$-series, and in Section 4.2 we define a regularized multiple Eisenstein series (Definition \ref{4_8})
$$G^\sh_{n_1,\ldots,n_r}(q):=\big( \mathfrak{z}^{\sh}\otimes \mathfrak{g}^\sh \big) \circ \Delta(I(n_1,\ldots,n_r)) \in\C[[q]] \quad (n_1,\ldots,n_r\in\Z_{\ge1}) .$$
It follows from the definition that the $q$-series $G^\sh_{n_1,\ldots,n_r}(q)$ ($n_1,\ldots,n_r\in\Z_{\ge1}$) satisfy the shuffle product formulas.
We will prove that $G^\sh_{n_1,\ldots,n_r}(q)$ coincides with the Fourier expansion of $G_{n_1,\ldots,n_r}(\tau)$ when $n_1,\ldots,n_r\ge2$ and $q=e^{2\pi \sqrt{-1}\tau}$ (Theorem \ref{4_10}).
Then, combining the shuffle product of $G^\sh$'s and the harmonic product of $G$'s yields the double shuffle relation for multiple Eisenstein series, which is the second main result of this paper (proved in Section 4.2).
\begin{theorem}\label{1_2}
The restricted finite double shuffle relations hold for $G^\sh_{n_1,\ldots,n_r}(q)$ ($n_1,\ldots,n_r\in\Z_{\ge1})$.
\end{theorem}

The organization of this paper is as follows.
In section 2, the Fourier expansion of the multiple Eisenstein series $G_{n_1,\ldots,n_r}(\tau)$ is considered.
In section 3, we first recall the regularized multiple zeta value and Hopf algebras of iterated integrals introduced by Goncharov. 
Then we define the map $\mathfrak{z}^{\sh}$ that assigns regularized multiple zeta value to formal iterated integrals. 
We also present the formula expressing $\Delta(I(n_1,\ldots,n_r))$ as a certain algebraic combination of $I(k_1,\ldots,k_i)$'s, and finally proves Theorem \ref{1_1}.
Section 4 gives the definition of the algebra homomorphism $\mathfrak{g}^{\sh}$ and proves double shuffle relations for multiple Eisenstein series.
A future problem with the dimension of the space of $G^\sh$'s will be discussed in the end of this section.

\section*{Acknowledgement}

This work was started during a visit of the second author at the Unversit\"at Hamburg in June 2014 and supported by the Research Training Group 1670 ``Mathematics Inspired by String Theory and Quantum Field Theory" at the Unversit\"at Hamburg.
The authors would like to thank Herbert Gangl, Masanobu Kaneko and Ulf K\"uhn for very useful comments. 
The second author was partially supported by Basic Science Research Program through the National Research Foundation of Korea (NRF) funded by the Ministry of Education (2013053914). 
The authors thank the reviewer for their thorough helpful and constructive comments, which significantly contributed to improving the manuscript.

\section{The Fourier expansion of multiple Eisenstein series}

\subsection{Multiple Eisenstein series}

In this subsection, we define the multiple Eisenstein series and compute its Fourier expansion.

Recall the computation of the Fourier expansion of $G_{n_1}(\tau)$, which is well-known (see also \cite[Section 7]{GKZ}):
\begin{align*} 
G_{n_1}(\tau)&=\sum_{0\prec l\tau+m} \frac{1}{(l\tau+m)^{n_1}}=\sum_{m>0} \frac{1}{m^{n_1}} + \sum_{l>0}\sum_{m\in\Z} \frac{1}{(l\tau+m)^{n_1}}\\
&= \zeta(n_1) + \frac{(-2\pi \sqrt{-1})^{n_1}}{(n_1-1)!} \sum_{n>0} \sigma_{n_1-1}(n) q^n ,
\end{align*}
where $\sigma_k(n)=\sum_{d\mid n}d^{k}$ is the divisor function and $ q=e^{2\pi \sqrt{-1}\tau}$.
Here for the last equality we have used the Lipschitz formula
\begin{equation}\label{eq2_1} 
\sum_{m\in\Z}\frac{1}{(\tau+m)^{n_1}} =\frac{(-2\pi \sqrt{-1})^{n_1} }{(n_1-1)!} \sum_{0<c_1} c_1^{n_1-1} q^{c_1} \quad (n_1\ge2).
\end{equation}
When $n_1=2$, the above computation (the second equality) can be justified by using a limit argument which in general is treated in Definition \ref{2_1} below.
We remark that the function $G_{n_1}(\tau)$ is a modular form of weight $n_1$ for ${\rm SL}_2(\Z)$ when $n_1$ is even $(>2)$ ($G_2(\tau)$ is called the quasimodular form) and a non-trivial holomorphic function even if $n_1$ is odd.

The following definition enables us to compute the Fourier expansion of $G_{n_1,\ldots,n_r}(\tau)$ for integers $n_1,\ldots,n_r\ge2$ and coincides with the iterated multiple sum \eqref{eq1_1} when the series defining \eqref{eq1_1} converges absolutely, i.e. $n_1,\ldots,n_{r-1}\ge2$ and $n_r\ge3$.

\begin{definition}\label{2_1}
For integers $n_1,\ldots,n_r\ge2$, we define the holomorphic function $G_{n_1,\ldots,n_r}(\tau)$ on the upper half-plane called the multiple Eisenstein series by
\begin{align*} G_{n_1,\ldots,n_r}(\tau) &:= \lim_{L\rightarrow \infty} \lim_{M\rightarrow \infty}   \sum_{\substack{0\prec \lambda_1\prec \cdots\prec \lambda_r\\ \lambda_i\in\Z_L \tau+\Z_M}} \frac{1}{\lambda_1^{n_1}\cdots \lambda_r^{n_r}}\\
&=  \lim_{L\rightarrow \infty} \lim_{M\rightarrow \infty}   \sum_{\substack{0\prec (l_1\tau+m_1)\prec \cdots\prec (l_r\tau+m_r)\\ -L \le l_1,\ldots,l_r \le L\\ -M\le m_1,\ldots,m_r \le M }} \frac{1}{(l_1\tau+m_1)^{n_1}\cdots (l_r\tau+m_r)^{n_r}},
\end{align*}
where we set $\Z_M=\{-M,-M+1,\ldots,-1,0,1,\ldots,M-1,M\}$ for an integer $M>0$. 
\end{definition}

The Fourier expansion of $G_{n_1,\ldots,n_r}(\tau)$ for integers $n_1,\ldots,n_r\ge2$ is obtained by splitting up the sum into $2^r$ terms, which was first done in \cite{GKZ} for the case $r=2$ and in \cite{Bach} for the general case (they use the opposite convention, so that the $\lambda_i$'s are ordered by $\lambda_1\succ\cdots \succ\lambda_r\succ 0$).
To describe each term we introduce the holomorphic function $G_{n_1,\ldots,n_r}(w_1 \cdots w_r;\tau) $ on the upper half-plane below.
For convenience, we express the set $P$ of positive elements in the lattice $\Z \tau + \Z$ as the disjoint union of two sets
\begin{align*}
P_\x &:= \left\{ l\tau + m \in\Z \tau + \Z \mid l = 0 \wedge m > 0  \right\} ,\\
P_\y &:= \left\{ l\tau + m \in \Z \tau + \Z \mid l > 0  \right\} ,
\end{align*}
i.e. $P_\x$ are the lattice points on the positive real axis, $P_\y$ are the lattice points in the upper half-plane and $P=P_\x \cup P_\y$. 
We notice that $\lambda_1 \prec \lambda_2 $ is equivalent to $ \lambda_2 - \lambda_1 \in P$. 
Let us denote by $\{ \x,  \y\}^*$ the set of all words consisting of letters $\x$ and $\y$.
For integers $n_1,\ldots,n_r\ge2$ and a word $w_1 \cdots w_r\in \{\x,\y\}^*$ ($w_i\in\{\x,\y\}$) we define
\begin{align*}
 G_{n_1,\ldots,n_r}(w_1 \cdots w_r)&=G_{n_1,\ldots,n_r}(w_1 \cdots w_r;\tau) \\
 &:=\lim_{L\rightarrow \infty} \lim_{M\rightarrow \infty} \sum_{\substack{\lambda_1 - \lambda_{0} \in P_{w_1}\\\vdots \\ \lambda_r - \lambda_{r-1} \in P_{w_r} \\ \lambda_1 , \dots, \lambda_r \in \Z_L \tau +\Z_M}} \frac{1}{\lambda_1^{n_1}\cdots \lambda_r^{n_r}} ,
 \end{align*}
where $\lambda_0 := 0$.
Note that in the above sum, adjoining elements $\lambda_i-\lambda_{i-1}=(l_i-l_{i-1})\tau+(m_i-m_{i-1}),\ldots,\lambda_j-\lambda_{j-1}=(l_j-l_{j-1})\tau+(m_j-m_{j-1})$ are in $P_\x$ (i.e. $w_i=\cdots=w_j=\x$ with $i\le j$) if and only if they satisfy $m_{i-1}<m_i<\cdots<m_j$ with $l_{i-1}=l_i=\cdots=l_j$ (since $(l-l')\tau+(m-m')\in P_\x$ if and only if $l=l'$ and $m>m'$), and hence the function $G_{n_1,\ldots,n_r}(w_1\cdots w_r)$ is expressible in terms of the following function:
\[ \Psi_{n_1,\ldots,n_r}(\tau) = \sum_{-\infty<m_1<\cdots<m_r<\infty} \frac{1}{(\tau+m_1)^{n_1}\cdots (\tau+m_r)^{n_r}},\]
which was studied thoroughly in \cite{Boulliot}.
In fact, as is easily seen that the series defining $\Psi_{n_1,\ldots,n_r}(\tau)$ converges absolutely when $n_1,\ldots,n_r\ge2$, we obtain the following expression:
\begin{equation}\label{eq2_2}
\begin{aligned}
&G_{n_1,\ldots, n_r}(w_1\cdots w_r)\\
&= \zeta(n_1,\ldots,n_{t_1-1})\sum_{0<l_1<\cdots<l_h} \Psi_{n_{t_1},\ldots,n_{t_2-1}}(l_1\tau)\cdots \Psi_{n_{t_h},\ldots,n_r}(l_h\tau),
\end{aligned}
\end{equation}
where $0<t_1<\cdots<t_h<r+1$ describe the positions of $\y$'s in the word $w_1\cdots w_r$, i.e.  
\begin{equation*}
w_1\cdots w_r=\underbrace{\x\cdots \x}_{t_1-1}\y \underbrace{\x \cdots \x}_{t_2-t_1-1} \y \x \cdots \y \underbrace{\x \cdots \x}_{t_h-t_{h-1}-1} \y \underbrace{\x \cdots \x}_{r-t_h},
\end{equation*}
and $\zeta(n_1,\ldots,n_{t_1-1})=1$ when $t_1=1$.

We remark that the above expression of words gives a one-to-one correspondence between words of length $r$ in $\{\x,\y\}^*$ and the ordered subsets of $\{1,2,\ldots,r\}$.
As we will use later, this correspondence is written as the map
\begin{equation}\label{eq2_3}
\begin{aligned} 
\rho: \{\x,\y\}^*_r&\longrightarrow 2^{\{1,2,\ldots,r\}}\\
w_1\cdots w_r&\longmapsto \{t_1,\ldots,t_h\},
\end{aligned}
\end{equation}
where $\{\x,\y\}^*_r$ is the set of words of length $r$ and $2^{\{1,2,\ldots,r\}}$ is the set of all subsets of $\{1,2,\ldots,r\}$.
For instance, we have $\rho(\x\y\x\y\x^{r-4}) =\{2,4\}$ and $\rho(\x^r)=\{\emptyset\}$.

\begin{proposition}\label{2_2}
For integers $n_1,\ldots,n_{r}\ge2$, we have
\[ G_{n_1,\ldots,n_r} (\tau) = \sum_{w_1,\ldots,w_r\in\{\x,\y\}} G_{n_1,\ldots,n_r}(w_1 \cdots w_r).\]
\end{proposition}

\begin{proof}
For $\lambda_1,\ldots,\lambda_r\in \Z_L \tau + \Z_M$, the condition $0\prec \lambda_1\prec \cdots\prec \lambda_r$ is by definition equivalent to $\lambda_i - \lambda_{i-1} \in P = P_\x \cup P_\y$ for all $1\le i\le r-1$ (recall $\lambda_0=0$).
Since $ \lambda_i - \lambda_{i-1}$ can be either in $P_\x$ or in $P_\y$ we complete the proof.
\end{proof}

\begin{example}\label{2_3}
In the case of $r=2$, one has for $n_1\ge2,n_2\ge3$
\begin{align*}
 G_{n_1,n_2}(\tau) &=  \sum_{\substack{0\prec \lambda_1\prec \lambda_2 \\\lambda_1,\lambda_2\in\Z \tau + \Z }} \lambda_1^{-n_1}\lambda_2^{-n_2} = \sum_{\substack{\lambda_1-\lambda_0 \in P\\ \lambda_2-\lambda_1\in P\\\lambda_1,\lambda_2\in\Z \tau + \Z }} \lambda_1^{-n_1}\lambda_2^{-n_2}\\
& =\bigg( \sum_{\substack{ \lambda_1-\lambda_0\in P_\x \\ \lambda_2-\lambda_1\in P_\x \\ \lambda_1,\lambda_2\in\Z \tau + \Z}} +\sum_{\substack{\lambda_1-\lambda_0\in P_\x \\ \lambda_2-\lambda_1\in P_\y \\ \lambda_1,\lambda_2\in\Z \tau + \Z}} + \sum_{\substack{\lambda_1-\lambda_0\in P_\y \\ \lambda_2-\lambda_1\in P_\x \\ \lambda_1,\lambda_2\in\Z \tau + \Z}}+\sum_{\substack{\lambda_1-\lambda_0\in P_\y \\ \lambda_2-\lambda_1\in P_\y \\ \lambda_1,\lambda_2\in\Z \tau + \Z}} \bigg) \lambda_1^{-n_1}\lambda_2^{-n_2}\\
&= G_{n_1,n_2}(\x\x) + G_{n_1,n_2}(\x\y)+G_{n_1,n_2}(\y\x) + G_{n_1,n_2}(\y\y).
\end{align*}
\end{example}

\subsection{Computing the Fourier expansion}

In this subsection, we give a Fourier expansion of $G_{n_1,\ldots,n_r}(w_1 \cdots w_r)$.

Let us define the $q$-series $g_{n_1,\ldots,n_r}(q)$ for integers $n_1,\ldots,n_r\ge1$ by
\begin{equation}\label{eq2_4}
g_{n_1,\ldots,n_r}(q) = \frac{(-2\pi \sqrt{-1})^{n_1+\cdots+n_r}}{(n_1-1)!\cdots(n_r-1)!} \sum_{\substack{0<d_1<\cdots<d_r\\c_1,\ldots,c_r\in\N\\d_1,\ldots,d_r\in\N}} c_1^{n_1-1}\cdots c_r^{n_r-1}q^{c_1d_1+\cdots+c_rd_r},
\end{equation}
which divided by $(-2\pi \sqrt{-1})^{n_1+\cdots+n_r}$ was studied in \cite{BK}.
We remark that since $g_{n_1}(q)$ is the generating series of the divisor function $\sigma_{n_1-1}(n)$ up to a scalar factor, the coefficient of $q^n$ in the $q$-series $g_{n_1,\ldots,n_r}(q)$ is called the multiple divisor sum in \cite{BK} with the opposite convention:
\[ \sigma_{n_1,\ldots,n_r}(n)=\sum_{\substack{c_1d_1+\cdots+c_rd_r=n\\0<d_1<\cdots<d_r\\c_1,\ldots,c_r\in \N\\d_1,\ldots,d_r\in\N}} c_1^{n_1-1}\cdots c_r^{n_r-1},\]
which is regarded as a multiple version of the divisor sum (we do not discuss their properties in this paper). 
We will investigate an algebraic structure related to the $q$-series $g_{n_1,\ldots,n_r}(q)$ in a subsequent paper. 

To give the Fourier expansion of $G_{n_1,\ldots,n_r}(w_1,\ldots,w_r)$, we need the following lemma.

\begin{lemma}\label{2_4}
For integers $n_1,\ldots,n_r\ge2$, we have
\begin{equation*}
\begin{aligned}
 \sum_{q=1}^r\sum_{\substack{k_1+\cdots+k_r=n_1+\cdots+n_r\\k_i\ge n_i,k_q=1}} & \bigg( (-1)^{n_q+k_{q+1}+\cdots+k_r} \prod_{\substack{j=1\\j\neq q}}^{r} \binom{k_j-1}{n_j-1} \\
&\times\zeta(k_{q-1},k_{q-2},\ldots,k_1)\zeta(k_{q+1},k_{q+2},\ldots,k_r)\bigg)=0,
\end{aligned}
\end{equation*} 
where $\zeta(n_1,\ldots,n_r)=1$ when $r=0$.
\end{lemma}

\begin{proof}
This was shown by using an iterated integral expression of multiple zeta values in \cite[Section 5.5]{Boulliot} (his notations $\mathcal{T}e^{n_r,\ldots,n_1}(z)$ and $\mathcal{Z}e^{n_r,\ldots,n_1}$ correspond to our $\Psi_{n_1,\ldots,n_r}(z) $ and $\zeta(n_1,\ldots,n_r)$, respectively).
We remark that he proved the identities in Lemma \ref{2_4} for $n_1,\ldots,n_r\ge1$ with $n_1,n_r\ge2$.
\end{proof}

\begin{proposition}\label{2_5}
For integers $n_1,\ldots,n_r\ge2$ and a word $w_1\cdots w_r\in \{\x,\y\}^*$, we set $N_{t_m}=n_{t_m}+\cdots+n_{t_{m+1}-1}$ for $m\in\{1,\ldots,h\}$ where $\{t_1,\ldots,t_h\}=\rho(w_1\cdots w_r)$ given by \eqref{eq2_3} and $t_{h+1}=r+1$.
Then the function $G_{n_1,\ldots, n_r}(w_1\cdots w_r;\tau)$ has the following Fourier expansion:
\begin{align*}
&G_{n_1,\ldots, n_r}(w_1\cdots w_r)= \zeta(n_1,\ldots,n_{t_1-1}) \\
&\times\sum_{\substack{t_1\le q_1\le t_2-1\\t_2\le q_2\le t_3-1\\\vdots\\ t_h\le q_h\le r}}  \sum_{\substack{k_{t_1}+\cdots+k_{t_{2}-1}=N_{t_1}\\k_{t_2}+\cdots+k_{t_3-1}=N_{t_2}\\
\vdots\\ k_{t_h}+\cdots+k_r=N_{t_h}\\ k_{t_1},k_{t_1+1},\ldots,k_r\ge2}}  \bigg\{(-1)^{\sum_{m=1}^h (N_{t_m}+n_{q_m}+k_{q_m+1}+k_{q_m+2}+\cdots+k_{q_{m+1}-1})}\\
&\times \bigg(\prod_{\substack{j=t_1\\j\neq q_1,\ldots,q_h}}^{r} \binom{k_j-1}{n_j-1} \bigg) \bigg(\prod_{m=1}^h \zeta(\underbrace{k_{q_m-1},\ldots,k_{t_m}}_{q_m-t_m})\zeta(\underbrace{k_{q_m+1},\ldots,k_{t_{m+1}-1}}_{t_{m+1}-q_m-1})\bigg) \\
&\times g_{k_{q_1},\ldots,k_{q_h}}(q)\bigg\},
\end{align*}
where $q=e^{2\pi \sqrt{-1}\tau}$, $\zeta(n_1,\ldots,n_r)=g_{n_1,\ldots,n_r}(q)=1$ whenever $r=0$ and $\prod_{\substack{j=t_1\\j\neq q_1,\ldots,q_h}}^{r} \binom{k_j-1}{n_j-1} =1$ when the product is empty, i.e. when $\{t_1,t_1+1,\ldots,r\}= \{q_1,\ldots,q_h\}$.
\end{proposition}

\begin{proof}
Put $N=n_1+\cdots+n_r$. 
Using the partial fraction decomposition
\begin{align*}
&\frac{1}{(\tau+m_1)^{n_1}\cdots (\tau+m_r)^{n_r}} \\
&= \sum_{q=1}^r\sum_{\substack{k_1+\cdots+k_r=N\\k_1,\ldots,k_r\ge1}} \left( \prod_{j=1}^{q-1}\frac{\binom{k_j-1}{n_j-1}}{(m_q-m_j)^{k_j}}\right) \frac{(-1)^{N+n_q}}{(\tau+m_q)^{k_q}} \left( \prod_{j=q+1}^r \frac{(-1)^{k_j}\binom{k_j-1}{n_j-1}}{(m_j-m_q)^{k_j}}  \right),
\end{align*}
we obtain
\begin{equation}\label{eq2_5}
\begin{aligned}
\Psi_{n_1,\ldots,n_r}(\tau)=&  \sum_{q=1}^r\sum_{\substack{k_1+\cdots+k_r=N\\k_1,\ldots,k_r\ge1}} \bigg( (-1)^{N+n_q+k_{q+1}+\cdots+k_r} \prod_{\substack{j=1\\j\neq q}}^{r} \binom{k_j-1}{n_j-1} \\
&\times\zeta(\underbrace{k_{q-1},k_{q-2},\ldots,k_1}_{q-1})\Psi_{k_q}(\tau) \zeta(\underbrace{k_{q+1},k_{q+2},\ldots,k_r}_{r-q})\bigg),
\end{aligned}
\end{equation}
where the implied interchange of order of summation is justified because the binomial coefficient $\binom{k_i-1}{n_i-1}$ vanishes if $k_1=1$ or $k_r=1$ and by Lemma \ref{2_4} the coefficient of $\Psi_1(\tau)$ is zero.
Using the Lipschitz formula \eqref{eq2_1} we easily find that
\begin{equation}\label{eq2_6}
g_{n_1,\ldots,n_r}(q) = \sum_{0 < d_1 <\dots <d_r} \Psi_{n_1}(d_1\tau) \dots \Psi_{n_r}(d_r\tau)
\end{equation}
for integers $n_1,\ldots,n_r\ge2$.
Thus, combining the above formulas \eqref{eq2_5} and \eqref{eq2_6} with \eqref{eq2_2}, we have the desired formula. 
\end{proof}

We remark that the formula \eqref{eq2_5}, which in the case of $r=2$ was done in \cite[Proof of Theorem 6]{GKZ}, is found in \cite[Theorem 3]{Boulliot} and holds when $n_1,\ldots,n_r\ge1$ with $n_1,n_r\ge2$, but we use only the formula \eqref{eq2_5} for $n_1,\ldots,n_r\ge2$ in this paper.

Let us illustrate a few examples.

\begin{example}
When $r=1$, we have for $n_1>1$ 
\begin{align*}
G_{n_1}(\x;\tau)&=\zeta(n_1)\ \mbox{and} \ G_{n_1}(\y;\tau)=\sum_{0<l} \Psi_{n_1}(l\tau) = g_{n_1}(q),
\end{align*}
and hence
\[G_{n_1}(\tau) = \zeta(n_1)+g_{n_1}(q).\]
\end{example}

\begin{example}\label{2_6}
We compute the case $r=2$, which was carried out in \cite{GKZ}.
From \eqref{eq2_2} and \eqref{eq2_6}, it follows
\begin{align*}
G_{n_1,n_2}(\x\x) &= \zeta(n_1,n_2),\\
G_{n_1,n_2} ( \x\y) &=  \zeta(n_1) \sum_{0<l} \Psi_{n_2}(l\tau) =\zeta(n_1) g_{n_2}(q), \\
 G_{n_1,n_2}(\y\y) &=  \sum_{0<l_1<l_2} \Psi_{n_1}(l_1\tau)\Psi_{n_2}(l_2\tau)=g_{n_1,n_2}(q),
\end{align*}
and using \eqref{eq2_5}, we have 
\[ G_{n_1,n_2}(\y\x) = \sum_{0<l} \Psi_{n_1,n_2}(l\tau) =\sum_{\substack{k_1+k_2=n_1+n_2\\k_1,k_2\ge2}} {\tt b}_{n_1,n_2}^{k_1} \zeta(k_1)g_{k_2}(q),\]
where for integers $n,n',k>0$ we set
\begin{equation}\label{eq2_7}
{\tt b}_{n,n'}^{k} = (-1)^{n} \binom{k-1}{n-1} +(-1)^{k-n'} \binom{k-1}{n'-1}.
\end{equation}
Thus the Fourier expansion of $G_{n_1,n_2}(\tau)$ is given by
\begin{equation}\label{eq2_8}
 G_{n_1,n_2}(\tau) = \zeta(n_1,n_2)+\sum_{\substack{k_1+k_2=n_1+n_2\\k_1,k_2\ge2}} \big( \delta_{n_1,k_1} + {\tt b}_{n_1,n_2}^{k_1} \big) \zeta(k_1)g_{k_2}(q) + g_{n_1,n_2}(q),
\end{equation}
where $\delta_{n,k}$ is the Kronecker delta.
\end{example}

\begin{example}\label{2_7}
For the future literature, we present the Fourier expansion of $G_{n_1,n_2,n_3}(\tau)$ with $n_1,n_2,n_3\ge2$:
\begin{align*}
&G_{n_1,n_2,n_3}(\tau)\\
& =\zeta (n_1,n_2,n_3)+\zeta(n_1,n_2)g_{n_3}(q)+\zeta(n_1)g_{n_2,n_3}(q)+g_{n_1,n_2,n_3}(q)\\
&+ \sum_{\substack{k_1+k_2+k_3=n_1+n_2+n_3\\ k_1,k_2,k_3\ge2}} \bigg\{\left( \delta_{n_3,k_3}{\tt b}_{n_1,n_2}^{k_1}+\delta_{n_1,k_2}{\tt b}_{n_2,n_3}^{k_1} \right) \zeta(k_1)g_{k_2,k_3}(q) \\
& +\left((-1)^{n_1+k_3}\binom{k_2-1}{n_3-1} + (-1)^{n_1+n_2} \binom{k_2-1}{n_1-1} \right)  \binom{k_1-1}{n_2-1} \zeta(k_1,k_2)g_{k_3}(q) \\
&+\left( (-1)^{n_1+n_3+k_2} \binom{k_1-1}{n_1-1}\binom{k_2-1}{n_3-1}  + \delta_{k_1,n_1} {\tt b}_{n_2,n_3}^{k_2} \right) \zeta(k_1)\zeta(k_2)g_{k_3}(q) \bigg\}   .
\end{align*}
\end{example}

\section{A relationship between multiple Eisenstein series and the Goncharov coproduct}

\subsection{Regularized multiple zeta values}

In this subsection, we recall the regularized multiple zeta value with respect to the shuffle product defined in \cite{IKZ}.
We first review an iterated integral expression of the multiple zeta value due to Kontsevich and Drinfel'd, and then recall the algebraic setup of multiple zeta values given by Hoffman.

We denote by $\omega_0(t) = \frac{dt}{t}$ and $\omega_1(t) = \frac{dt}{1-t}$ holomorphic 1-forms on the smooth manifold $\mathbb{P}_{\C}^1\backslash\{0,1,\infty\}$.
For integers $n_1,\ldots,n_{r-1}\ge1$ and $n_r\ge2$ with $N=n_1+\cdots+n_r$, the multiple zeta value $\zeta(n_1,\ldots,n_r)$ is expressible as an iterated integral on the smooth manifold $\mathbb{P}^1_{\C}\backslash\{0,1,\infty\}$:
\begin{equation}\label{eq3_1}
\begin{aligned}
\zeta(n_1,\ldots,n_r) &=\mathop {\int \cdots \int}_{0<t_1<t_2<\cdots<t_N<1} \omega_{a_1}(t_1)\wedge \omega_{a_2}(t_2)\wedge \cdots \wedge \omega_{a_N}(t_{N}),
\end{aligned} 
\end{equation}
where $a_i=1$ if $i\in \{1,n_1+1,n_1+n_2+1,\ldots,n_1+\cdots+n_{r-1}+1\}$ and $a_i=0$ otherwise.

Let $\h= \Q\langle e_0 ,e_1 \rangle$ be the non-commutative polynomial algebra in two indeterminates $e_0$ and $e_1$, and $\h^1:=\Q+e_1\h$ and $\h^0:=\Q+e_1\h e_0$ its subalgebras. 
Set 
\[ y_n := e_1e_0^{n-1}=e_1\underbrace{e_0\cdots e_0}_{n-1} \]
for each positive integer $n>0$.
It is easily seen that the subalgebra $\h^1$ is freely generated by $y_n$'s $(n\ge1)$ as a non-commutative polynomial algebra: 
\[ \h^1 = \Q\langle y_1,y_2,y_3,\ldots\rangle .\]
We define the shuffle product, a $\Q$-bilinear product on $\h$, inductively by
\begin{align*}
 &u  w \ \sh \ v  w' = u ( w\ \sh \ v w') + v(u w\ \sh \ w') ,
\end{align*}
with the initial condition $w\ \sh \ 1=1\ \sh \ w=w$ $(1\in\Q)$, where $w,w' \in\h$ and $u,v\in\{e_0,e_1\}$.
This provides the structures of commutative $\Q$-algebras for spaces $\h,\h^1$ and $\h^0$ (see \cite{Reu}), which we denote by $\h_\sh,\h_\sh^1$ and $\h_\sh^0$ respectively.
By taking the iterated integral \eqref{eq3_1}, with the identification $w_i(t)\leftrightarrow e_i \ (i\in\{0,1\})$, one can define an algebra homomorphism
\begin{align*}
Z:\h^0_\sh&\longrightarrow \R\\
y_{n_1}\cdots y_{n_r} &\longmapsto \zeta(n_1,\ldots,n_r) \quad (n_r>1)
\end{align*}
with $Z(1)=1$, since it is shown by K.T.~Chen \cite{Chen} that the iterated integral \eqref{eq3_1} satisfies the shuffle product formulas.
By \cite[Proposition 1]{IKZ}, there is a $\Q$-algebra homomorphism 
\[ Z^\sh:\h^1_\sh\rightarrow\R[T]\]
which is uniquely determined by the properties that $Z^{\sh}\big|_{\h^0_\sh} = Z$ and $Z^{\sh}(e_1)=T$.
We note that the image of the word $y_{n_1}\cdots y_{n_r}$ in $\h^1_\sh$ under the map $Z^\sh$ is a polynomial in $T$ whose coefficients are expressed as $\Q$-linear combinations of multiple zeta values.

\begin{definition}\label{3_1}
The regularized multiple zeta value, denoted by $\zeta^{\sh}(n_1,\ldots,n_r)$, is defined as the constant term of $Z^\sh (y_{n_1}\cdots y_{n_r})$ in $T$:
\[ \zeta^{\sh}(n_1,\ldots,n_r) := Z^\sh (y_{n_1}\cdots y_{n_r})\big|_{T=0} .\]
\end{definition}

For example, we have $\zeta^{\sh}(2,1)=-2\zeta(1,2)$ and 
\begin{equation}\label{eq3_2}
\zeta^\sh(n_1,\ldots,n_r)=\zeta(n_1,\ldots,n_r) \quad (n_r\ge2).
\end{equation}

\subsection{Hopf algebras of iterated integrals}

In this subsection, we recall Hopf algebras of formal iterated integrals introduced by Goncahrov.

In his paper \cite[Section 2]{G}, Goncharov considered a formal version of the iterated integrals
\begin{equation}\label{eq3_3} 
\int_{a_0}^{a_{N+1}} \frac{dt_1}{t_1-a_1} \cdots \int_{t_{N-2}}^{a_{N+1}} \frac{dt_{N-1}}{t_{N-1}-a_{N-1}} \int_{t_{N-1} }^{a_{N+1}}\frac{dt_N}{t_N-a_N} 
  \quad (a_i\in \C).
 \end{equation}
He proved that the space $\mathcal{I}_\bullet(S)$ generated by such formal iterated integrals carries a Hopf algebra structure.
Let us recall the definition of the space $\mathcal{I}_\bullet(S)$.

\begin{definition}\label{3_2}
Let $S$ be a set. 
Let us denote by $\mathcal{I}_{\bullet}(S)$ the commutative graded algebra over $\Q$ generated by the elements
\[  \mathbb{I}(a_0;a_1,\ldots,a_N;a_{N+1}), \quad  N\ge0,\ a_i\in S \]
with degree $N$ which are subject to the following relations.
\begin{enumerate}
\setlength{\itemsep}{-5pt}
\item[{\rm (I1)}] For any $a,b\in S$, the unit is given by $\mathbb{I}(a;b):=\mathbb{I}(a;\emptyset;b)=1$.
\item[{\rm (I2)}] The product is given by the shuffle product: for all integers $N,N'\ge0$ and $a_i\in S$, one has
\begin{align*}
& \mathbb{I}(a_0;a_1,\ldots,a_N;a_{N+N'+1})\mathbb{I}(a_0;a_{N+1},\ldots,a_{N+N'};a_{N+N'+1})\\
&= \sum_{\sigma\in\Sigma(N,N')} \mathbb{I}(a_0;a_{\sigma^{-1}(1)},\ldots,a_{\sigma^{-1}(N+N')};a_{N+N'+1}),
\end{align*}
where $\Sigma(N,N')$ is the set of $\sigma$ in the symmetric group $\mathfrak{S}_{N+N'}$ such that $\sigma(1)<\cdots<\sigma(N)$ and $\sigma(N+1)<\cdots<\sigma(N+N')$.
\item[{\rm (I3)}] The path composition formula holds: for any $N\ge0$ and $a_i,x\in S$, one has 
\[ \mathbb{I}(a_0;a_1,\ldots,a_N;a_{N+1}) = \sum_{k=0}^N \mathbb{I}(a_0;a_1,\ldots,a_k;x) \mathbb{I}(x;a_{k+1},\ldots,a_N;a_{N+1}).\]
\item[{\rm (I4)}] For $N\ge1$ and $a_i,a\in S$, $\mathbb{I} (a;a_1,\ldots,a_N;a)=0$.
\end{enumerate}
\end{definition}

We remark that the element $\mathbb{I}(a_0;a_1,\ldots,a_N;a_{N+1})$ is an analogue of the iterated integral \eqref{eq3_3}, since by K.T.~Chen \cite{Chen} iterated integrals satisfy (I1) to (I4) when the integral converges.

The Goncharov coproduct $\Delta: \mathcal{I}_{\bullet}(S) \rightarrow \mathcal{I}_{\bullet}(S)\otimes \mathcal{I}_{\bullet}(S)$ is defined by 
\begin{equation}\label{eq3_4}
\begin{aligned}
 \Delta \big(\mathbb{I}(a_0&;a_1,\ldots,a_N;a_{N+1})\big) \\
=&  \sum_{\substack{0\le k \le N\\i_0=0<i_1<\cdots<i_k<i_{k+1}=N+1}} \prod_{p=0}^k \mathbb{I}(a_{i_p};a_{i_p+1},\ldots,a_{i_{p+1}-1};a_{i_{p+1}})  \\
&\otimes \mathbb{I}(a_0;a_{i_1},\ldots,a_{i_k};a_{N+1}),
\end{aligned}
\end{equation}
for any $N\ge0$ and $a_i\in S$, and then extending by $\Q$-linearity.
The formula \eqref{eq3_4} can be found in \cite{B1} (see Eq.~(2.18)), originally given by Goncharov (see \cite[Eq.~(27)]{G}) with the factors interchanged.

\begin{proposition}{\rm \cite[Proposition 2.2]{G}}\label{3_3}
The Goncharov coproduct $\Delta$ gives $\mathcal{I}_{\bullet}(S)$ the structure of a commutative graded Hopf algebra, where the counit $\varepsilon$ is determined by the condition that it kills $\mathcal{I}_{>0}(S)$. 
\end{proposition}

We remark that the antipode $\mathcal{S}$ of the above Hopf algebra is uniquely and inductively determined by the definition (see \cite[Lemma A.1]{G}).
For example, since $\Delta(\mathbb{I}(a_0;a_1;a_2))=\mathbb{I}(a_0;a_1;a_2)\otimes1+1\otimes \mathbb{I}(a_0;a_1;a_2)$ for any $a_0,a_1,a_2\in S$, we have 
\[ \mathcal{S}(\mathbb{I}(a_0;a_1;a_2))+\mathbb{I}(a_0;a_1;a_2)=0= u \circ \varepsilon (\mathbb{I}(a_0;a_1;a_2)) ,\]
where $u:\Q\rightarrow \mathcal{I}_\bullet(S)$ is the unit.
In general, the formula is obtained from the fact that
\begin{align*}
& \mathcal{S}(\mathbb{I}(a_0;a_1,\ldots,a_N;a_{N+1}))+\mathbb{I}(a_0;a_1,\ldots,a_N;a_{N+1}) \\
&= \mbox{a $\Z$-linear combination of products of $\mathbb{I}$'s of degree $<N$},
\end{align*} 
which we do not develop the precise formula in this paper.

\subsection{Formal iterated integrals and regularized multiple zeta values}

In this subsection, we define the map $\mathfrak{z}^\sh$ described in the introduction.
Hereafter, we restrict only the Hopf algebra $\mathcal{I}_\bullet:= \mathcal{I}_{\bullet}(S)$ to the case with $S=\{0,1\}$.

Consider the quotient algebra
\[ \mathcal{I}_{\bullet}^1:=\mathcal{I}_{\bullet} / \mathbb{I}(0;0;1)\mathcal{I}_{\bullet} .\]
It is easy to verify that $\mathbb{I}(0;0;1)$ is primitive, i.e. $\Delta(\mathbb{I}(0;0;1))=1\otimes \mathbb{I}(0;0;1) + \mathbb{I}(0;0;1)\otimes 1$.
Thus the ideal $\mathbb{I}(0;0;1)\mathcal{I}_{\bullet}$ generated by $\mathbb{I}(0;0;1)$ in the $\Q$-algebra $\mathcal{I}_{\bullet}$ becomes a Hopf ideal, and hence the quotient map $\mathcal{I}_{\bullet} \rightarrow \mathcal{I}_{\bullet}^1$ induces a Hopf algebra structure on the quotient algebra $\mathcal{I}_{\bullet}^1$.
Let us denote by 
\[ I(a_0;a_1,\ldots,a_N;a_{N+1}) \in \mathcal{I}_{\bullet}^1\]
an image of $\mathbb{I}(a_0;a_1,\ldots,a_N;a_{N+1})$ in $\mathcal{I}_{\bullet}^1$ and by the same symbol $\Delta$ the induced coproduct on $\mathcal{I}_{\bullet}^1$ given by the same formula as \eqref{eq3_4} replacing $\mathbb{I}$ with $I$.
As a result, we have the following proposition which we will use later.

\begin{proposition}\label{3_4}
The coproduct $\Delta :\mathcal{I}_{\bullet}^1 \rightarrow \mathcal{I}_{\bullet}^1\otimes \mathcal{I}_{\bullet}^1$ is an algebra homomorphism, where the product on $\mathcal{I}_{\bullet}^1\otimes \mathcal{I}_{\bullet}^1$ is defined in the standard way by $(w_1\otimes w_2)(w_1'\otimes w_2')=w_1w_1'\otimes w_2w_2'$ and the product on each summand $\mathcal{I}^1_{\bullet}$ is given by the shuffle product {\rm (I2)}.
\end{proposition}

We remark that dividing $\mathcal{I}_{\bullet}$ by $\mathbb{I}(0;0;1)\mathcal{I}_{\bullet}$ can be viewed as a regularization for ``$\int_0^1 dt/t=-\log(0)$" which plays a role as $\mathbb{I}(0;0;1)$ in the evaluation of iterated integrals.
For example, one can write $I(0;0,1,0;1)=-2I(0;1,0,0;1)$ in $\mathcal{I}^1_{\bullet}$ since it follows $\mathbb{I}(0;0,1,0;1)=\mathbb{I}(0;0;1)\mathbb{I}(0;1,0;1)-2\mathbb{I}(0;1,0,0;1)$, and this computation corresponds to taking the constant term of $\int_{\varepsilon}^1 \frac{dt_1}{t_1}\int_{\varepsilon}^{t_1} \frac{dt_2}{1-t_2}\int_{\varepsilon}^{t_2}\frac{dt_3}{t_3}$ as a polynomial of $\log(\varepsilon)$ and letting $\varepsilon\rightarrow 0$.

By the standard calculation about the shuffle product formulas, we obtain more identities in the space $\mathcal{I}_{\bullet}^1$ (see \cite[p.955]{B1}).

\begin{enumerate}
\item For $n\ge1$ and $a,b\in\{0,1\}$, we have 
\begin{equation}\label{eq3_5}
I(a;\underbrace{0,\ldots,0}_n;b)=0. 
\end{equation}
\item For integers $n\ge0,n_1,\ldots,n_r\ge1$, we have
\begin{equation}\label{eq3_6}
\begin{aligned}
&I(0;\underbrace{0,\ldots,0}_n,\underbrace{1,0,\ldots,0}_{n_1},\ldots,\underbrace{1,0,\ldots,0}_{n_r};1)\\
&= (-1)^n \sum_{\substack{k_1+\cdots+k_r=n_1+\cdots +n_r+n\\ k_1,\ldots,k_r\ge1}} \bigg(\prod_{j=1}^{r}\binom{k_j-1}{n_j-1} \bigg) I (k_1,\ldots,k_r),
\end{aligned}
\end{equation}
where we set 
\[I(n_1,\ldots,n_r):=I(0;\underbrace{1,0,\ldots,0}_{n_1},\ldots,\underbrace{1,0,\ldots,0}_{n_r};1).\]
\end{enumerate}

In order to define the map $\mathfrak{z}^\sh$ as a $\Q$-linear map, we give a linear basis of the space $\mathcal{I}_{\bullet}^1$ first.

\begin{proposition}\label{3_5}
The set of elements $\{I(n_1,\ldots,n_r)\mid r\ge0,n_i\ge1\}$ is a linear basis of the space $\mathcal{I}_{\bullet}^1$.
\end{proposition}

\begin{proof}
Recall the result of Goncharov \cite[Proposition 2.1]{G}: for each integer $N\ge0$ and $a_0,\ldots,a_{N+1}\in\{0,1\}$ one has
\begin{equation}\label{eq3_7} 
\mathbb{I}(a_0;a_1,\ldots,a_N;a_{N+1}) = (-1)^N \mathbb{I}(a_{N+1};a_N,\ldots,a_1;a_0),
\end{equation}
which essentially follows from (I3) and (I4).
Then, we find that the collection 
\[ \{\mathbb{I}(0;a_1,\ldots,a_N;1)\mid N\ge0,a_i\in\{0,1\}\}\]
 forms a linear basis of the linear space $\mathcal{I}_{\bullet}$, since none of the relations (I1) to (I4) yield $\Q$-linear relations among them.
Combining this with \eqref{eq3_6}, we obtain the desired basis. 
\end{proof}

\begin{definition}\label{3_6}
Let $\mathfrak{z}^\sh:\mathcal{I}_{\bullet}^1\rightarrow \R$ be the $\Q$-linear map given by
\begin{align*}
\mathfrak{z}^\sh:\mathcal{I}_{\bullet}^1&\longrightarrow \R\\
 I(n_1,\ldots,n_r) &\longmapsto \zeta^{\sh}(n_1,\ldots,n_r)
\end{align*}
and $\mathfrak{z}^\sh(1)=1$.
\end{definition}

\begin{proposition}\label{3_7}
The map $\mathfrak{z}^{\sh} : \mathcal{I}_{\bullet}^1\rightarrow \R$ is an algebra homomorphism.
\end{proposition}
\begin{proof}
By Proposition \ref{3_5}, we find that the $\Q$-linear map $\h^1_\sh\rightarrow \mathcal{I}_{\bullet}^1$ given by $e_{a_1}\cdots e_{a_N}\mapsto I(0;a_1,\ldots, a_N;1)$ is an isomorphism between $\Q$-algebras.
Then the result follows from the standard fact that the map $Z^{\sh}\big|_{T=0} :\h_\sh^1 \rightarrow \R$ given by $ y_{n_1}\cdots y_{n_r} \mapsto \zeta^{\sh}(n_1,\ldots,n_r)$ is an algebra homomorphism.
\end{proof}

\subsection{Computing the Goncharov coproduct}

In this subsection, we rewrite the Goncharov coproduct $\Delta$ for $I(n_1,\ldots,n_r)$ in terms of $I(k_1,\ldots,k_i)$'s.
Although one can compute the formula by using Propositions \ref{3_8} and \ref{3_11}, we do not give explicit formulas for $\Delta(I(n_1,\ldots,n_r))$ in general.
We present an explicit formula for only $\Delta(I(n_1,n_2,n_3))$ in the end of this subsection.

To describe the formula, it is convenient to use the algebraic setup.
Let $\h':=\langle e_0,e_1,e_0',e_1'\rangle$ be the non-commutative polynomial algebra in four indeterminates $e_0,e_1,e_0'$ and $e_1'$.
For integers $0<i_1<i_2<\cdots <i_k<N+1$ \ $(0\le k\le N)$, we define 
\[ {\bf e}_{i_1,\ldots,i_k} (a_1,\ldots,a_N):= e_{a_1}\cdots e_{a_{i_1-1}} \left( \prod_{p=1}^{k-1} e_{a_{i_p}}' e_{a_{i_p+1}}\cdots e_{a_{i_{p+1}-1}} \right) e_{a_{i_k}}'e_{a_{i_k+1}}\cdots e_{a_{N}}\]
with each $a_j=0$ or $1$, where the product $\prod_{p=1}^{k-1}$ means the concatenation product.
Let 
\[\varphi : \h' \rightarrow \mathcal{I}_\bullet^1 \otimes\mathcal{I}_\bullet^1\]
 be the $\Q$-linear map that assigns to each word ${\bf e}_{i_1,\ldots,i_k} (a_1,\ldots,a_N)$ the factor of the right-hand side of the equation \eqref{eq3_4} with $a_0=0$ and $a_{N+1}=1$:
\begin{equation*}
\begin{aligned}
  &\varphi ({\bf e}_{i_1,\ldots,i_k} (a_1,\ldots,a_N))\\
  & = \prod_{p=0}^k I(a_{i_p};a_{i_p+1},\ldots,a_{i_{p+1}-1};a_{i_{p+1}}) \otimes I(0;a_{i_1},\ldots,a_{i_k};1),
  \end{aligned}
\end{equation*}
where we set $a_{i_0}=0$ and $a_{i_{k+1}}=1$.
For example, we have $\varphi({\bf e}_{2,3}(a_1,\ldots,a_4))=\varphi(e_{a_1}e_{a_2}' e_{a_3}' e_{a_4})= I(0;a_1;a_2)I(a_2;a_3)I(a_3;a_4;1)\otimes I(0;a_2,a_3;1)$.

In the rest of this subsection, for integers $n_1,\ldots,n_r\ge1$ with $N=n_1+\cdots+n_r$, we set 
\[\{a_1,\cdots, a_N\}=\{1,\underbrace{0,\ldots,0}_{n_1-1},1,\underbrace{0,\ldots,0}_{n_2-1},\ldots ,1,\underbrace{0,\ldots,0}_{n_r-1}\},\]
and write ${\bf e}_{i_1,\ldots,i_k}(n_1,\ldots,n_r):={\bf e}_{i_1,\ldots,i_k}(a_1,\ldots,a_N)$.
We note that $a_j=1$ if $j$ lies in the set 
\[ P_{n_1,\ldots,n_r}:=\{1,n_1+1,\ldots,n_1+\cdots+n_{r-1}+1\},\]
and $a_j=0$ otherwise.

Using the above notations, one has
\begin{equation}\label{eq3_8}
\Delta(I(n_1,\ldots,n_r)) =\sum_{k=0}^N \sum_{ 0<i_1<\cdots<i_k<N+1} \varphi({\bf e}_{i_1,\ldots,i_k} (n_1,\ldots,n_r)).
\end{equation}
To compute \eqref{eq3_8}, we split the right-hand side of \eqref{eq3_8} into $2^r$ sums of $\psi_{n_1,\ldots,n_r}(w_1\cdots w_r)$ defined below.

Define the map
\[ \iota_{n_1,\ldots,n_r} : 2^{\{1,2,\ldots,r\}} \longrightarrow 2^{P_{n_1,\ldots,n_r}}\]
given by $1\mapsto 1$ and $i\mapsto n_1+\cdots+n_{i-1}+1 \ (2\le i\le r)$, where for a set $X$ the set of all subsets of $X$ is denoted by $2^X$.
It follows that the map $\iota_{n_1,\ldots,n_r}$ is bijective.
Let
\begin{equation*}
\overline{\rho}_{n_1,\ldots,n_r}:= \iota_{n_1,\ldots,n_r} \circ \rho : \{\x,\y\}^\ast_r \longrightarrow 2^{P_{n_1,\ldots,n_r}},
\end{equation*}
where $\rho$ is defined in \eqref{eq2_3}.
The map $\overline{\rho}$ is also bijective. 
For instance, we have $\overline{\rho}_{n_1,\ldots,n_r}(\y\x\y\x^{r-3})=\{1,n_1+n_2+1\}$ and $\overline{\rho}_{n_1,\ldots,n_r}(\x^r)=\{\emptyset\}$.
With the map $\overline{\rho}_{n_1,\ldots,n_r}$, for integers $n_1,\ldots,n_r\ge1$ and a word $w_1\cdots w_r$ ($w_i\in\{\x,\y\}$) we define 
\begin{equation}\label{eq3_9}
\psi_{n_1,\ldots,n_r}(w_1\cdots w_r) := \sum_{k=h}^N \sum_{\substack{0<i_1<\cdots<i_k<N+1\\ \{i_1,\ldots,i_k\}\cap P_{n_1,\ldots,n_r}= \overline{\rho}_{n_1,\ldots,n_r}(w_1\cdots w_r)}} \varphi({\bf e}_{i_1,\ldots,i_k} (n_1,\ldots,n_r)) ,
\end{equation}
where $h$ is the number of $\y$'s in the word $w_1\cdots w_r$ (i.e. $h=\deg_\y (w_1\cdots w_r)$).

\begin{proposition}\label{3_8}
For integers $n_1,\ldots,n_r\ge1$, we have
\[ \Delta\big(I(n_1,\ldots,n_r)\big) =  \sum_{w_1,\ldots,w_r\in\{\x,\y\}}  \psi_{n_1,\ldots,n_r}(w_1\cdots w_r).\]
\end{proposition}

\begin{proof}
For the word ${\bf e}_{i_1,\ldots,i_k}(n_1,\ldots,n_r)$, we denote by $h$ the number of $e_1'$'s in the prime symbols $e_{a_{i_1}}',\ldots,e_{a_{i_k}}'$, i.e. $h=\deg_{e_1'} ({\bf e}_{i_1,\ldots,i_k}(n_1,\ldots,n_r))$.
Since $a_j=1$ if and only if $j\in P_{n_1,\ldots,n_r}$, we have $h=\sharp (\{i_1,\ldots,i_k\}\cap P_{n_1,\ldots,n_r})$.
We notice that $h$ can be chosen from $\{0,1,\ldots,\min\{r,k\}\}$ for each $k$.
With this, the formula \eqref{eq3_8} can be written in the form
\begin{align*}
 \eqref{eq3_8} &= \sum_{k=0}^N \sum_{h=0}^{\min\{r,k\}} \sum_{\substack{0<i_1<\cdots<i_k<N+1\\ \sharp (\{i_1,\ldots,i_k\}\cap P_{n_1,\ldots,n_r})=h}}   \varphi({\bf e}_{i_1,\ldots,i_k}(n_1,\ldots,n_r))\\
&=  \sum_{h=0}^r \sum_{k=h}^{N} \sum_{\substack{0<i_1<\cdots<i_k<N+1\\ \sharp (\{i_1,\ldots,i_k\}\cap P_{n_1,\ldots,n_r})=h}}   \varphi({\bf e}_{i_1,\ldots,i_k}(n_1,\ldots,n_r)).
\end{align*}
By specifying the subset of $P_{n_1,\ldots,n_r}$ with length $h$, the above third sum can be split into the following sums:
\begin{align*}
 \eqref{eq3_8}&=\sum_{h=0}^r \sum_{k=h}^N\sum_{\substack{\{j_1,\ldots,j_h\}\subset P_{n_1,\ldots,n_r}\\j_1<\cdots<j_h}}  \sum_{\substack{0<i_1<\cdots<i_k<N+1\\ \{i_1,\ldots,i_k\}\cap P_{n_1,\ldots,n_r}= \{j_1,\ldots,j_h\}}} \varphi({\bf e}_{i_1,\ldots,i_k}(n_1,\ldots,n_r))\\
 &=\sum_{h=0}^r \sum_{\substack{\{j_1,\ldots,j_h\}\subset P_{n_1,\ldots,n_r}\\j_1<\cdots<j_h}}  \sum_{k=h}^N\sum_{\substack{0<i_1<\cdots<i_k<N+1\\ \{i_1,\ldots,i_k\}\cap P_{n_1,\ldots,n_r}= \{j_1,\ldots,j_h\}}} \varphi({\bf e}_{i_1,\ldots,i_k}(n_1,\ldots,n_r))\\
&=\sum_{h=0}^r \sum_{\substack{w_1,\ldots,w_r\in\{\x,\y\}\\ \deg_\y w_1\cdots w_r =h}} \psi_{n_1,\ldots,n_r}(w_1\cdots w_r)\\
&=\sum_{w_1,\ldots,w_r\in\{\x,\y\}}  \psi_{n_1,\ldots,n_r}(w_1\cdots w_r),
\end{align*}
which completes the proof.
\end{proof}

We express \eqref{eq3_9} as algebraic combinations of $I(k_1,\ldots,k_i)$'s.
To do this, we extract possible nonzero terms from the right-hand side of \eqref{eq3_9} by using (I4). 
For a positive integer $n$, we define $\eta_0 (n)$ as the sum of all words of degree $n-1$ consisting of $e_0$ and a consecutive $e_0'$:
\[ \eta_0 (n) = \sum_{\substack{\alpha+k+\beta=n\\\alpha,\beta\ge0\\k\ge1}} e_0^{\alpha} (e_0')^{k-1} e_0^{\beta}.\]

\begin{proposition}\label{3_9}
For integers $n_1,\ldots,n_r\ge1$ and a word $w_1\cdots w_r$ of length $r$ in $\{\x,\y\}^*$, we have
{\footnotesize \begin{equation*}
\begin{aligned}
&\psi_{n_1,\ldots,n_r}(w_1\cdots w_r)=\\
&\sum_{\substack{t_1\le q_1\le t_2-1\\t_2\le q_2\le t_3-1\\\vdots\\ t_h\le q_h\le r}} \varphi\Big(  y_{n_1} \cdots y_{n_{t_1-1}} \prod_{m=1}^h \big(e_1' \underbrace{e_0^{n_{t_m}-1} y_{n_{t_m+1}}\cdots y_{n_{q_{m-1}}} e_1}_{{\rm degree\ in}\ e_1 = q_m-t_m }\eta_0(n_{q_m}) y_{n_{q_m+1}}\cdots y_{n_{t_{m+1}-1}}\big)\bigg),
\end{aligned}
\end{equation*}}
where $\{t_1,\ldots,t_h\}=\rho(w_1\cdots w_r)$ given by \eqref{eq2_3}, $t_{h+1}=r+1$ and the product $\prod_{m=1}^h$ means the concatenation product of words.
\end{proposition}

\begin{proof}
It follows $\psi_{n_1,\ldots,n_r}(\x^r)= \varphi(y_{n_1}\cdots y_{n_r})$, so we consider the case $h>0$ which means the number of $\y$'s in $w_1\cdots w_r$ is greater than 0.
We note that the sum defining \eqref{eq3_9} runs over all words $c_{a_1}\cdots c_{a_N}$ ($c_{a_i}\in\{e_{a_i},e_{a_i}'\}$) with $k$ ($h\le k\le N$) prime symbols whose positions of $e_1'$'s are placed on the set $\{j_1,\ldots,j_h\}=\overline{\rho}_{n_1,\ldots,n_r}(w_1\cdots w_r)$:
\begin{align*}
\eqref{eq3_9} &= \sum_{k=h}^N \sum_{\substack{w_0,w_1,\ldots,w_h\in \{e_0,e_1,e_0'\}^*\\ \deg_{e_0'} (w_0w_1\cdots w_h) = k-h\\ \deg (w_0)= j_1-1 \\ \deg (e_1'w_m)=j_{m+1}-j_m \ (1\le m\le h) \\ j_{h+1}=N+1}} \varphi \Big(  w_0 \prod_{m=1}^h \big(e_1' w_m \big) \Big).
\end{align*}

We find by (I4) that $\varphi({\bf e}_{i_1,\ldots,i_k}(n_1,\ldots,n_r))$ is 0 whenever $a_{i_1}=0$ (notice that $a_0=0$ and $a_1=1$).
This implies that if the degree of the above $w_0$ in the letter $e_0'$ is greater than 0, then $\varphi \Big(  w_0 \prod_{m=1}^h \big(e_1' w_m \big) \Big)=0$.
For a word $w\in\h'$, we also find $\varphi(w)=0$ if $w$ contains a subword of the form $e_0' v e_0'$ with $v \in\h \ (v\neq \emptyset)$, i.e. $w=w_1 e_0' ve_0'  w_2$ for some $w_1,w_2\in\h'$, because the factor of the left-hand side of $\varphi(w)$ involves $I(0;v;0)$ which by (I4) is 0.
This implies that the above second sum regarding $w_m$ ($1\le m\le h$) of the form $w_m=w_1 e_0' ve_0'  w_2$ with $v \in\{e_0,e_1\}^* \ (v\neq \emptyset)$ and $w_1,w_2\in\{e_0,e_1,e_0'\}^*$ can be excluded.
Thus, the possible nonzero terms in \eqref{eq3_9}, sieved out by (I4), occur if $w_0=y_{n_1}\cdots y_{n_{t_1-1}}$ and $w_m$ is written in the form
\[ \underbrace{e_0^{n_{t_m}-1} y_{n_{t_m+1}}\cdots y_{n_{q_{m-1}}} e_1}_{{\rm degree\ in}\ e_1 = q_m-t_m }e_0^\alpha (e_0')^{k} e_0^{\beta} y_{n_{q_m+1}}\cdots y_{n_{t_{m+1}-1}}, \]
where $q_m\in \{t_m,t_m+1,\ldots, t_{m+1}-1\}$, $\alpha,\beta,k\in\Z_{\ge0}$ with $\alpha+k+\beta=n_{q_m}-1$ and $\{t_1,\ldots,t_h\}=\rho(w_1\cdots w_r)$.
This completes the proof.
\end{proof}

Before giving an explicit formula for $\psi_{n_1,\ldots,n_r}(w_1\cdots w_r)$, we illustrate examples for $r=1$ and $2$.

\begin{example}
By definition, for $n_1\ge1$ one has $\psi_{n_1}(\x)=\varphi(e_1e_0^{n_1-1})=I(n_1)\otimes1$.
By \eqref{eq3_5}, we obtain
\[\psi_{n_1}(\y)= \sum_{\substack{\alpha+k+\beta=n_1\\ \alpha,\beta\ge0 \\ k\ge1}}  \varphi(e_1' e_0^{\alpha} (e_0')^{k-1} e_0^{\beta})=1\otimes I(n_1).\]
Thus, we get
\[\Delta(I(n_1)) = I(n_1)\otimes 1 + 1\otimes I(n_1) .\]
\end{example}

\begin{example}\label{3_10}
It follows $\psi_{n_1,n_2}(\x\x)=\varphi(e_1e_0^{n_1-1}e_1 e_0^{n_2-1})=I(n_1,n_2)\otimes1$.
By \eqref{eq3_5} one can compute
\begin{align*}
\psi_{n_1,n_2}(\x\y)&=\sum_{\substack{\alpha+k+\beta=n_2\\ \alpha,\beta\ge0 \\ k\ge1}}  \varphi(e_1e_0^{n_1-1}e_1' e_0^{\alpha} (e_0')^{k-1} e_0^{\beta})=I(n_1)\otimes I(n_2),\\
\psi_{n_1,n_2}(\y\y)&=\sum_{\substack{\alpha_1+k_1+\beta_1=n_1\\ \alpha_1,\beta_1\ge0 \\ k_1\ge1}}\sum_{\substack{\alpha_2+k_2+\beta_2=n_2\\ \alpha_2,\beta_2\ge0 \\ k_2\ge1}} \varphi(e_1' e_0^{\alpha_1} (e_0')^{k_1-1} e_0^{\beta_1} e_1' e_0^{\alpha_2} (e_0')^{k_2-1} e_0^{\beta_2})\\
&=1\otimes I(n_1,n_2),
\end{align*}
and using \eqref{eq3_6} and \eqref{eq3_7} we have
\begin{align*}
\psi_{n_1,n_2}(\y\x)&= \sum_{\substack{\alpha+k+\beta=n_1\\ \alpha,\beta\ge0 \\ k\ge1}} \varphi( e_1' e_0^{\alpha} (e_0')^{k-1} e_0^{\beta}e_1e_0^{n_2-1} ) + \sum_{\substack{\alpha+k+\beta=n_2\\ \alpha,\beta\ge0 \\ k\ge1}}\varphi( e_1' e_0^{n_1-1}  e_1 e_0^{\alpha} (e_0')^{k-1} e_0^{\beta})\\
&= \sum_{\substack{k_1+k_2=n_1+n_2\\ k_1,k_2\ge1}} {\tt b}_{n_1,n_2}^{k_1} I(k_1)\otimes I(k_2),
\end{align*}
where ${\tt b}_{n,n'}^k$ is defined in \eqref{eq2_7}.
Therefore by Proposition \ref{3_8} we have
\begin{equation}\label{eq3_10}
 \begin{aligned}
& \Delta(I(n_1,n_2)) \\
 &= I(n_1,n_2)\otimes 1+\sum_{\substack{k_1+k_2=n_1+n_2\\k_1,k_2\ge1}} \big( \delta_{n_1,k_1} + {\tt b}_{n_1,n_2}^{k_1} \big) I(k_1)\otimes I(k_2) + 1\otimes I(n_1,n_2).
 \end{aligned}
\end{equation}
\end{example}

\begin{proposition}\label{3_11}
For integers $n_1,\ldots,n_r\ge2$ and a word $w_1\cdots w_r\in \{\x,\y\}^*$, we set $N_{t_m}=n_{t_m}+n_{t_m+1}+\cdots+n_{t_{m+1}-1}$ for each $m\in\{1,2,\ldots,h\}$, where $\{t_1,\ldots,t_h\}=\rho(w_1\cdots w_r)$ and $t_{h+1}=r+1$.
Then we have
\begin{align*}
&\psi_{n_1,\ldots,n_r}(w_1\cdots w_r) = (I(n_1,\ldots,n_{t_1-1}) \otimes 1)\\
&\times \sum_{\substack{t_1\le q_1\le t_2-1\\t_2\le q_2\le t_3-1\\\vdots\\ t_h\le q_h\le r}}  \sum_{\substack{k_{t_1}+\cdots+k_{t_{2}-1}=N_{t_1}\\k_{t_2}+\cdots+k_{t_3-1}=N_{t_2}\\ \vdots\\ k_{t_h}+\cdots+k_r=N_{t_h}\\ k_i\ge1}} \bigg\{ (-1)^{\sum_{m=1}^h (N_{t_m}+n_{q_m}+k_{q_m+1}+k_{q_m+2}+\cdots+k_{q_{m+1}-1})}\\
&\times \bigg(\prod_{\substack{j=t_1\\j\neq q_1,\ldots,q_h}}^{r} \binom{k_j-1}{n_j-1} \bigg) \bigg(\prod_{m=1}^h I(\underbrace{k_{q_m-1},\ldots,k_{t_m}}_{q_m-t_m}) I(\underbrace{k_{q_m+1},\ldots,k_{t_{m+1}-1}}_{t_{m+1}-q_m-1})\bigg) \\
&\otimes I(k_{q_1},\ldots,k_{q_h})\bigg\},
\end{align*}
where $\prod_{\substack{j=t_1\\j\neq q_1,\ldots,q_h}}^{r} \binom{k_j-1}{n_j-1} =1$ when $\{t_1,t_1+1,\ldots,r\}= \{q_1,\ldots,q_h\}$.
\end{proposition}

\begin{proof}
This can be verified by applying the identities \eqref{eq3_5}, \eqref{eq3_6} and \eqref{eq3_7} to the formula in Proposition \ref{3_9}.
\end{proof}

\begin{example}\label{3_12}
For the future literature, we present an explicit formula for $\Delta(I(n_1,n_2,n_3))$ obtained from Propositions \ref{3_8} and \ref{3_11}:
\begin{align*}
&\Delta(I(n_1,n_2,n_3))\\
& =I (n_1,n_2,n_3)\otimes 1+I(n_1,n_2)\otimes I(n_3)+I(n_1)\otimes I(n_2,n_3)+1\otimes I(n_1,n_2,n_3)\\
&+ \sum_{\substack{k_1+k_2+k_3=n_1+n_2+n_3\\ k_1,k_2,k_3\ge1}} \bigg\{\left( \delta_{n_3,k_3}{\tt b}_{n_1,n_2}^{k_1}+\delta_{n_1,k_2}{\tt b}_{n_2,n_3}^{k_1} \right) I(k_1)\otimes I(k_2,k_3) \\
& +\left((-1)^{n_1+k_3}\binom{k_2-1}{n_3-1} + (-1)^{n_1+n_2} \binom{k_2-1}{n_1-1} \right)  \binom{k_1-1}{n_2-1} I(k_1,k_2)\otimes I(k_3) \\
&+\left( (-1)^{n_1+n_3+k_2} \binom{k_1-1}{n_1-1}\binom{k_2-1}{n_3-1}  + \delta_{k_1,n_1} {\tt b}_{n_2,n_3}^{k_2} \right) I(k_1)I(k_2)\otimes I(k_3) \bigg\}   .
\end{align*}
\end{example}

\subsection{Proof of Theorem \ref{1_1}}

We now give a proof of Theorem \ref{1_1}.
Recall the $q$-series $g_{n_1,\ldots,n_r}(q)$ defined in \eqref{eq2_4}.
Let 
\[\mathfrak{g}: \mathcal{I}_{\bullet}^1 \rightarrow \C[[q]]\]
be the $\Q$-linear map given by $\mathfrak{g}(I(n_1,\ldots,n_r))= g_{n_1,\ldots,n_r}(q)$ and $\mathfrak{g}(1)=1$.

\begin{proof}[Proof of Theorem \ref{1_1}]
Taking $\mathfrak{z}^{\sh}\otimes \mathfrak{g}$ for the explicit formula in Proposition \ref{3_11} and comparing this with Proposition \ref{2_5}, we have 
\[ \big(\mathfrak{z}^{\sh}\otimes \mathfrak{g}\big) \big(\psi_{n_1,\ldots,n_r}(w_1\cdots w_r)\big)= G_{n_1,\ldots,n_r}(w_1\cdots w_r) .\]
Here the second sum (relating to $k_i$) of the formula in Proposition \ref{3_11} differs from that of the formula in Proposition \ref{2_5}, but apparently it is the same because binomial coefficient terms allow us to take $k_i\ge n_i$ for $t_1\le i \le r$ with $i\neq q_1,\ldots,q_h$, and by Lemma \ref{2_4} it turns out that the coefficient of $g_{k_{q_1},\ldots,k_{q_h}}(q)$ becomes 0 if $k_{q_j}=1$ for some $1\le j\le h$. 
With this, the statement follows from Propositions \ref{2_2} and \ref{3_8}. 
\end{proof}

\section{The algebra of multiple Eisenstein series}

\subsection{The shuffle algebra consisting of certain $q$-series}

In this subsection, we define the map $\mathfrak{g}^\sh:\mathcal{I}_\bullet^1 \rightarrow \C[[q]]$ described in the introduction. 
We first introduce the power series $H\tbinom{n_1,\ldots,n_r}{x_1,\ldots,x_r}$ satisfying the harmonic product formulas.
Then, by using Hoffman's results, the power series $h(x_1,\ldots,x_r)$ (see \eqref{eq4_1}) satisfying the shuffle relation \eqref{eq4_3}, which is a variant of the shuffle relation (I2) (reformulated in \eqref{eq4_5}), is constructed.
Finally, we introduce the power series $g_\sh(x_1,\ldots,x_r)$ (see \eqref{eq4_4}) and prove in Theorem \ref{4_6} that their coefficients, which are $q$-series, satisfy the shuffle relations given in (I2).

Consider the iterated multiple sum
\begin{align*}
H\tbinom{n_1,\ldots,n_r}{x_1,\ldots,x_r} :=\sum_{0<d_1<\cdots<d_r} e^{d_1x_1}\left(\frac{q^{d_1}}{1-q^{d_1}}\right)^{n_1} \cdots e^{d_rx_r}\left(\frac{q^{d_r}}{1-q^{d_r}}\right)^{n_r},
\end{align*}
where $n_1,\ldots,n_r$ are positive integers and $x_1,\ldots,x_r$ are commutative variables, i.e. these are elements in the power series ring $\mathcal{K} [[x_1,\ldots,x_r]]$, where $\mathcal{K}=\Q[[q]]$.
It is easily seen that the power series $H\tbinom{n_1,\ldots,n_r}{x_1,\ldots,x_r}$ satisfies the harmonic product: for example, one has
\begin{align*}
H\tbinom{1}{x_1}H\tbinom{1}{x_2}&= \sum_{0<d_1,d_2} e^{d_1x_1} \frac{q^{d_1}}{1-q^{d_1}} e^{d_2x_2}\frac{q^{d_2}}{1-q^{d_2}} \\
&= \bigg( \sum_{0<d_1<d_2}+\sum_{0<d_2<d_1}+\sum_{0<d_1=d_2} \bigg)e^{d_1x_1} \frac{q^{d_1}}{1-q^{d_1}} e^{d_2x_2}\frac{q^{d_2}}{1-q^{d_2}}\\
&= H\tbinom{1,1}{x_1,x_2}+H\tbinom{1,1}{x_2,x_1}+H\tbinom{2}{x_1+x_2}.
\end{align*}

The power series $H\tbinom{n_1,\ldots,n_r}{x_1,\ldots,x_r}$ naturally appears in an expression of the generating series of $g_{n_1,\ldots,n_r}(q)$.

\begin{lemma}\label{4_1}
For any $r>0$, let
\[ g(x_1,\ldots,x_r):=\sum_{n_1,\ldots,n_r\ge1} \frac{g_{n_1,\ldots,n_r}(q)}{(-2\pi \sqrt{-1})^{n_1+\cdots+n_r} }x_1^{n_1-1}\cdots x_r^{n_r-1}.\]
Then, we have
\begin{equation*}
g(x_1,\ldots,x_r)= H\tbinom{1,\ldots,1,1}{x_r-x_{r-1},\ldots,x_2-x_1,x_1} .
\end{equation*}
\end{lemma}

\begin{proof}
When $r=2$ this was computed in the proof of Theorem 7 in \cite{GKZ} with the opposite convention. 
Our claim is its generalization, which can be easily shown. 
\end{proof}

For $r>0$, consider the power series
\begin{equation}\label{eq4_1}
 h(x_1,\ldots,x_r):=\sum_{m=1}^r\sum_{(i_1,i_2,\ldots,i_m)} \frac{1}{i_1!i_2!\cdots i_m!} H\tbinom{i_1,i_2,\ldots,i_m}{x_{i_1}',x_{i_2}',\ldots,x_{i_m}'},
\end{equation}
where the second sum runs over all decompositions of the integer $r$ as a sum of $m$ positive integers and the variables are given by $x_{i_1}'=x_1+\cdots+x_{i_1},x_{i_2}'=x_{i_1+1}+\cdots+x_{i_1+i_2},\ldots,x_{i_m}'=x_{i_1+\cdots+i_{m-1}+1}+\cdots+x_r$.
For example, we have 
\[ h(x)=H\tbinom{1}{x},\ h(x_1,x_2)=H\tbinom{1,1}{x_1,x_2}+\frac{1}{2}H\tbinom{2}{x_1+x_2}.\]

We shall prove that the power series $h(x_1,\ldots,x_r)$ satisfies the shuffle relation \eqref{eq4_3} below, which in the case of $r=2$ is given by 
\[ h(x_1)h(x_2) = h(x_1,x_2)+h(x_2,x_1).\]
To do this, we will reformulate the result of Hoffman \cite[Theorem 2.5]{H1} in accordance with our situation.

Let $\mathcal{U}$ be the non-commutative polynomial algebra over $\Q$ generated by non-commutative symbols $\tbinom{n}{z}$ indexed by $n\in \N$ and $z\in X:=\{ \sum_{i>0} a_i x_i \mid a_i\in\Z_{\ge0}$ is zero for almost all $i$'s $ \}$.
The word consisting of the concatenation products of letters $\tbinom{n_1}{z_1},\ldots, \tbinom{n_r}{z_r}$ is denoted by $\tbinom{n_1,\ldots,n_r}{z_1,\ldots,z_r} \ \big(=\tbinom{n_1}{z_1}\cdots \tbinom{n_r}{z_r}\big)$, for short.
The empty word is viewed as $1\in\Q$.
As usual, the harmonic product $\ast$ on $\mathcal{U}$ is inductively defined for $n,n'\in \N$, $z,z'\in X$ and words $w,w'$ in $\mathcal{U}$ by
\begin{align*}
 &\big(\tbinom{n}{z} \cdot w \big)\ast\big( \tbinom{n'}{z'} \cdot w'\big)\\
 &= \tbinom{n}{z} \cdot \big(w \ast \big(\tbinom{n'}{z'} \cdot w'\big)\big)+ \tbinom{n'}{z'} \cdot \big(\big( \tbinom{n}{z} \cdot w\big)  \ast   w'\big)+ \tbinom{n+n'}{z+z'} \cdot (w \ast w'),
\end{align*}
with the initial condition $w\ast 1=1\ast w=w$.
The shuffle product on $\mathcal{U}$ is defined in the same way as in $\sh$ on $\h=\Q\langle e_0,e_1 \rangle$, replacing the underlying vector space with $\mathcal{U}$.
Let us define the $\Q$-linear map ${\rm exp}:\mathcal{U}\rightarrow \mathcal{U}$ given for each linear basis $\tbinom{n_1,\ldots,n_r}{z_1,\ldots,z_r}$ by  
\begin{equation}\label{eq4_2}
{\rm exp} \left( \tbinom{n_1,\ldots,n_r}{z_1,\ldots,z_r} \right) =\sum_{m=1}^r\sum_{(i_1,i_2,\ldots,i_m)} \frac{1}{i_1!i_2!\cdots i_m!} \tbinom{n_{i_1}',n_{i_2}',\ldots,n_{i_m}'}{z_{i_1}',z_{i_2}',\ldots,z_{i_m}'},
\end{equation}
where the second sum runs over all decompositions of the integer $r$ as a sum of $m$ positive integers and the variables are given by $z_{i_1}'=z_1+\cdots+z_{i_1},z_{i_2}'=z_{i_1+1}+\cdots+z_{i_1+i_2},\ldots,z_{i_m}'=z_{i_1+\cdots+i_{m-1}+1}+\cdots+z_r$ and $n_{i_1}'=n_1+\cdots+n_{i_1},n_{i_2}'=n_{i_1+1}+\cdots+n_{i_1+i_2},\ldots,n_{i_m}'=n_{i_1+\cdots+i_{m-1}+1}+\cdots+n_r$.
This map was first considered by Hoffman (see \cite[p.52]{H1}) and called the exponential map.
Then, by \cite[Theorem 2.5]{H1}, we have the following proposition.

\begin{proposition}\label{4_2}
The exponential map gives the isomorphism 
\[{\rm exp} : \mathcal{U}_{\sh}\longrightarrow \mathcal{U}_{\ast}\]
as commutative $\Q$-algebras, where $\mathcal{U}_{\sh}$ (resp. $\mathcal{U}_{\ast}$) denotes the commutative algebra $\mathcal{U}$ equipped with the product $\sh$ (resp. $\ast$).
\end{proposition}

We now prove the shuffle relation for $h(x_1,\ldots,x_r)$'s.
We denote by $\mathfrak{S}_r$ the symmetric group of order $r$.
The group $\mathfrak{S}_r$ acts on $\mathcal{K}[[x_1,\ldots,x_r]]$ in the obvious way by $(f\big|\sigma)(x_1,\ldots,x_r)= f(x_{\sigma^{-1}(1)},\ldots,x_{\sigma^{-1}(r)})$, which defines a right action, i.e. $f\big|(\sigma\tau)=(f\big|\sigma)\big|\tau$.
This action extends to an action of the group ring $\Z[\mathfrak{S}_r]$ by linearity.

\begin{lemma}\label{4_3}
For any $r,s\ge1$, we have
\begin{equation}\label{eq4_3}
\begin{aligned}
&h(x_1,\ldots,x_r) h(x_{r+1},\ldots,x_{r+s}) = h(x_1,\ldots,x_{r+s})\big| sh_{r}^{(r+s)},
\end{aligned}
\end{equation}
where $sh_r^{(r+s)} := \sum_{\sigma\in\Sigma(r,s)}\sigma$ is in the group ring $\Z[\mathfrak{S}_{r+s}]$ and the set $\Sigma(r,s)$ is defined in the shuffle product formula {\rm (I2)}.
\end{lemma}

\begin{proof} 
Define the $\Q$-linear map 
\begin{equation*}
\begin{aligned}
 H:\mathcal{U}&\longrightarrow  \mathcal{R}:=\mathop {\varinjlim}_r \mathcal{K} [[x_1,\ldots,x_r]]\\
\tbinom{n_1,\ldots,n_r}{z_1,\ldots,z_r}&\longmapsto H\tbinom{n_1,\ldots,n_r}{z_1,\ldots,z_r}
\end{aligned}
\end{equation*}
for each linear basis $\tbinom{n_1,\ldots,n_r}{z_1,\ldots,z_r}$ of $\mathcal{U}$ with $H(1)=1$.
By combining the exponential map \eqref{eq4_2} with the map $H$, one easily finds that
\[  h(x_1,\ldots,x_r)= H\circ{\rm exp} \left( \tbinom{1,\ldots,1}{x_1,\ldots,x_r} \right). \]
Since the power series $H\tbinom{n_1,\ldots,n_r}{z_1,\ldots,z_r}$ satisfies the harmonic product, the algebra homomorphism $H:\mathcal{U}_\ast\rightarrow \mathcal{R}$ is obtained, and hence, by Proposition \ref{4_2}, the composition map $H\circ{\rm exp} : \mathcal{U}_\sh \rightarrow  \mathcal{R}$ is an algebra homomorphism.
Then, we have
\begin{align*}
h(x_1,\ldots,x_r)h(x_{r+1},\ldots,x_{r+s}) &= H\circ{\rm exp} \left( \tbinom{1,\ldots,1}{x_1,\ldots,x_r} \ \sh \ \tbinom{1,\ldots,1}{x_{r+1},\ldots,x_{r+s}} \right)\\
&=\sum_{\sigma\in \Sigma(r,s)} H\circ{\rm exp} \left( \tbinom{1,\ldots,1}{x_{\sigma^{-1}(1)},\ldots,x_{\sigma^{-1}(r+s)}} \right)\\
&= h(x_1,\ldots,x_{r+s}) \big| sh_r^{(r+s)},
\end{align*}
which completes the proof.
\end{proof}

For $r>0$, consider the power series
\begin{equation}\label{eq4_4}
g_{\sh}(x_1,\ldots,x_r) := h(x_r-x_{r-1},\ldots,x_2-x_1,x_1) .
\end{equation}
We shall prove that the $q$-series obtained from the coefficient of $x_1^{n_1-1}\cdots x_r^{n_r-1}$ in the power series $g_{\sh}(x_1,\ldots,x_r)$ satisfies the shuffle product formulas (I2).
For this, we use the generating series expression of the shuffle product formulas (I2) (this expression is found in \cite[Proof of Proposition 7]{IKZ} with the opposite convention).

\begin{proposition}\label{4_4}
For any $r>0$, let
\[ \mathbf{I}(x_1,\ldots,x_r) := \sum_{n_1,\ldots,n_r>0} I(n_1,\ldots,n_r) x_1^{n_1-1}\cdots x_r^{n_r-1} .\]
Then, the shuffle product formulas for $I(n_1,\ldots,n_r)$'s obtained in {\rm (I2)} is equivalent to  
\begin{equation}\label{eq4_5}
\mathbf{I}^{\sharp} (x_1,\ldots,x_r)\mathbf{I}^{\sharp}(x_{r+1},\ldots,x_{r+s}) = \mathbf{I}^{\sharp} (x_1,\ldots x_{r+s})\big| sh_r^{(r+s)},
\end{equation}
where the operator $\sharp$ is the change of variables defined by 
\[f^{\sharp}(x_1,\ldots,x_r)=f(x_1,x_1+x_2,\ldots,x_1+\cdots+x_r).\]
\end{proposition}

\begin{proof}
Computing the shuffle product formula (I2), one obtains
\begin{equation}\label{eq4_6}
I(n_1)I(n_2)=\sum_{\substack{k_1+k_2=n_1+n_2\\k_1,k_2\ge1}} \left(\binom{k_2-1}{n_1-1}+\binom{k_2-1}{n_2-1}\right) I(k_1,k_2),
\end{equation}
and hence we have 
\[  \mathbf{I}(x_1)  \mathbf{I}(x_2)  =  \mathbf{I}(x_2,x_1+x_2) + \mathbf{I}(x_1,x_1+x_2),\]
which coincides with the identity \eqref{eq4_5} for $r=s=1$.
The reminder (the case $r+s>2$) can be verified by induction.
\end{proof}

We now define the map $\mathfrak{g}^\sh:\mathcal{I}_\bullet^1 \rightarrow \C[[q]]$ and prove that this is an algebra homomorphism.

\begin{definition}\label{4_5}
Let us denote by $g^{\sh}_{n_1,\ldots,n_r}(q)$ the coefficient of $x_1^{n_1-1}\cdots x_r^{n_r-1}$ in the power series
\[  g_{\sh}(-2\pi\sqrt{-1} x_1,\ldots,-2\pi \sqrt{-1} x_r) = \sum_{n_1,\ldots,n_r>0}g^{\sh}_{n_1,\ldots,n_r}(q)x_1^{n_1-1}\cdots x_r^{n_r-1}. \]
With this, we define the $\Q$-linear map $\mathfrak{g}^{\sh}:\mathcal{I}_{\bullet}^1\rightarrow \C[[q]]$ for each linear basis $I(n_1,\ldots,n_r)$ of $\mathcal{I}_{\bullet}^1$ by 
$$\mathfrak{g}^{\sh}(I(n_1,\ldots,n_r))=g^{\sh}_{n_1,\ldots,n_r}(q)$$
and $\mathfrak{g}^{\sh}(1)=1$.
\end{definition}

\begin{theorem}\label{4_6}
The $q$-series $g^{\sh}_{n_1,\ldots,n_r}(q)$ satisfies the shuffle product formula {\rm (I2)}, namely, the map $\mathfrak{g}^{\sh}:\mathcal{I}_{\bullet}^1\rightarrow \C[[q]]$ is an algebra homomorphism.
\end{theorem}

\begin{proof}
By Proposition \ref{4_4}, it is sufficient to show that for any integers $r,s\ge1$ the power series $g_{\sh} (x_1,\ldots,x_{r+s})$ satisfies the relation 
\begin{equation*}
g_{\sh}^{\sharp} (x_1,\ldots,x_r)g_{\sh}^{\sharp}(x_{r+1},\ldots,x_{r+s}) = g_{\sh}^{\sharp} (x_1,\ldots x_{r+s})\big| sh_r^{(r+s)}.
\end{equation*}
For integers $r,s\ge1$, let 
\[ \rho_{r,s}=\begin{pmatrix} 1&2& \cdots & r & r+1 & \cdots &r+s-1& r+s \\ r&r-1& \cdots &1 & r+s & \cdots &r+2& r+1  \end{pmatrix}\in \mathfrak{S}_{r+s}.\] 
Applying the operator $\sharp$ to both sides of \eqref{eq4_4}, we obtain $g_{\sh}^{\sharp} (x_1,\ldots,x_r) = h(x_r,\ldots,x_1)$, and hence by Lemma \ref{4_3} we can compute 
\begin{align*}
g_{\sh}^{\sharp} (x_1,\ldots,x_r)g_{\sh}^{\sharp}(x_{r+1},\ldots,x_{r+s}) &= h(x_r,\ldots,x_1) h(x_{r+s},\ldots,x_{r+1}) \\
&= h(x_1,\ldots,x_r) h(x_{r+1},\ldots,x_{r+s}) \big| \rho_{r,s} \\&=  h(x_1,\ldots,x_{r+s})\big| sh_{r}^{(r+s)}\big| \rho_{r,s} \\
&= g_{\sh}^{\sharp} (x_1,\ldots,x_{r+s}) \big| \tau_{r+s} \big| sh_{r}^{(r+s)}\big| \rho_{r,s}  ,
\end{align*}
where we set $\tau_{r+s}=\begin{pmatrix} 1&2& \cdots & r+s \\ r+s& r+s-1& \cdots &1 \end{pmatrix}\in \mathfrak{S}_{r+s}$.
For any $\sigma \in \Sigma(r,s)$, one easily finds that $\tau_{r+s}\sigma\rho_{r,s} \in \Sigma(r,s)$, and hence
\[ g_{\sh}^{\sharp} (x_1,\ldots,x_{r+s})  \big| \tau_{r+s} \big| sh_{r}^{(r+s)}\big| \rho_{r,s} = g_{\sh}^{\sharp} (x_1,\ldots,x_{r+s})  \big| sh_{r}^{(r+s)},\]
which completes the proof.
\end{proof}

\begin{remark}\label{4_7}
From Theorem \ref{4_6}, we learn that the $q$-series $g_{n_1,\ldots,n_r}(q)$ does not satisfy the shuffle product formula (I2).
For example, by Theorem \ref{4_6} one has $g_\sh(x_1)g_\sh(x_2)= g_{\sh}(x_2,x_1+x_2)+g_{\sh}(x_1,x_1+x_2)$.
Since $ g_{\sh}(x)=g(x)$ and $g_{\sh}(x_1,x_2)=h(x_2-x_1,x_1)=g(x_1,x_2)+\frac{1}{2}H\tbinom{2}{x_2}$, one gets 
\begin{equation*}
g(x_1)g(x_2) = g(x_2,x_1+x_2)+g(x_1,x_1+x_2) + H\tbinom{2}{x_1+x_2},
\end{equation*}
which proves that the $q$-series $g_{n_1,n_2}(q)$ ($n_1,n_2\ge1$) do not satisfy the shuffle product formulas \eqref{eq4_6}, because $ H\tbinom{2}{x_1+x_2}\neq 0$.
\end{remark}

\subsection{Proof of Theorem \ref{1_2}}

In this subsection, we introduce the regularized multiple Eisenstein series $G^\sh_{n_1,\ldots,n_r}(q)$ for integers $n_1,\ldots,n_r\ge1$ as a $q$-series.
By relating $G^\sh$'s with $G$'s (Theorem \ref{4_10}), we show that the multiple Eisenstein series $G^\sh$'s satisfy the harmonic product (Theorem \ref{4_11}). 
Using these results, we finally prove the double shuffle relations for regularized multiple Eisenstein series (Theorem \ref{1_2}).

The regularized multiple Eisenstein series $G_{n_1,\ldots,n_r}^{\sh}(q)$ is defined as follows.

\begin{definition}\label{4_8}
For integers $n_1,\ldots,n_r\ge1$ we define the $q$-series $G_{n_1,\ldots,n_r}^{\sh}(q)$ by
\[ G_{n_1,\ldots,n_r}^{\sh}(q) = \big( \mathfrak{z}^{\sh}\otimes \mathfrak{g}^{\sh} \big) \circ \Delta(I(n_1,\ldots,n_r)) .\]
\end{definition}

\begin{example}\label{4_9}
For an integer $n\ge1$, we have
\[ G_{n}^\sh(q) = \zeta^\sh(n)+g^\sh_n(q) = \zeta^\sh(n)+g_n(q).\]
For integers $n_1,n_2\ge1$, by \eqref{eq3_10}, one has
\[ G_{n_1,n_2}^\sh(q) = \zeta^\sh(n_1,n_2)+\sum_{\substack{k_1+k_2=n_1+n_2\\k_1,k_2\ge1}} \big( \delta_{n_1,k_1} + {\tt b}_{n_1,n_2}^{k_1} \big) \zeta^\sh(k_1)g^\sh_{k_2}(q) + g^\sh_{n_1,n_2}(q),\]
where ${\tt b}_{n_1,n_2}^{k_1}$ is defined in \eqref{eq2_7}.
We remark that the above double Eisenstein series coincides with Kaneko's double Eisenstein series developed in \cite{Kaneko}.
\end{example}

We begin by showing a connection with the multiple Eisenstein series $G_{n_1,\ldots,n_r}(\tau)$, which can be regarded as an analogue of \eqref{eq3_2}.

\begin{theorem}\label{4_10}
For integers $n_1,\ldots,n_r\ge2$, with $q=e^{2\pi \sqrt{-1}\tau}$ we have
\[  G^{\sh}_{n_1,\ldots,n_r}(q) = G_{n_1,\ldots,n_r}(\tau).\]
\end{theorem}

\begin{proof}
This equation follows once we obtain the following identity: for integers $n_1,\ldots,n_r\ge2$
$$g^{\sh}_{n_1,\ldots,n_r}(q)=g_{n_1,\ldots,n_r}(q).$$
Combining \eqref{eq4_4} with \eqref{eq4_1}, we have
\begin{align*}
& g_{\sh}(x_1,\ldots,x_r) = \sum_{m=1}^r\sum_{(i_1,i_2,\ldots,i_m)} \frac{1}{i_1!i_2!\cdots,i_m!} H\tbinom{i_1,i_2,\ldots,i_m}{x_{i_1}'',x_{i_2}'',\ldots,x_{i_m}''},
 \end{align*}
where the second sum runs over all decompositions of the integer $r$ as a sum of $m$ positive integers and $x_{i_1}''= x_r-x_{r-i_1}, x_{i_2}'' =x_{r-i_1}-x_{r-i_1-i_2},\ldots,x_{i_m}''= x_{r-i_1-\cdots - i_{m-1}}$.
For any $n_1,\ldots,n_r\ge2$, since we have
\[ \mbox{coefficient of}\ x_1^{n_1-1}\cdots x_r^{n_r-1}\ \mbox{in}\ H\tbinom{i_1,i_2,\ldots,i_m}{x_{i_1}'',x_{i_2}'',\ldots,x_{i_m}''} =0 \]
whenever $i_j>1$ for some $j\in\{1,2,\ldots,m\}$,
we obtain
\begin{align*}
&\mbox{coefficient of $x_1^{n_1-1}\cdots x_r^{n_r-1}$ in $g_\sh (x_1,\ldots,x_r)$} \\
&= \mbox{coefficient of $x_1^{n_1-1}\cdots x_r^{n_r-1}$ in $H\tbinom{1,\ldots,1,1}{x_r-x_{r-1},\ldots,x_2-x_1,x_1}$},
\end{align*}
which by Lemma \ref{4_1} is equal to $g_{n_1,\ldots,n_r}(q)/(-2\pi \sqrt{-1})^{n_1+\cdots+n_r}$. 
This completes the proof.
\end{proof}

Let us prove the harmonic product for $G^\sh$'s.
The harmonic product $\ast$ on $\h^1$ is defined inductively by
\begin{align*}
 &y_{n_1}  w \ast y_{n_2}  w' = y_{n_1} ( w\ast y_{n_2} w') + y_{n_2}(y_{n_1} w\ast w') + y_{n_1+n_2} (w \ast  w'),
\end{align*}
and $w\ast 1=1\ast w=w$ for letters $y_{n_1},y_{n_2} \in\h^1$ and words $w,w'$ in $\h^1$, together with $\Q$-bilinearity. 
For each word $w\in\h^1$, the dual element of $w$ is denoted by $c_w\in(\h^1)^{\vee}={\rm Hom}(\h^1,\Q)$ such that $c_w(v)=\delta_{w,v}$ for any word $v\in \h^1$.
If $w$ is the empty word $\emptyset$, $c_w$ kills $\h^1_{>0}$ and $c_w(1)=1$.
Then, the harmonic product of the $q$-series $G^\sh_{n_1,\ldots,n_r}(q)$ with $n_1,\ldots,n_r\ge2$ is stated as follows:

\begin{theorem}\label{4_11}
For any words $w_1,w_2$ in $\{y_2,y_3,y_4,\ldots\}^{*}$, one has
\[ G_{w_1}^\sh (q) G_{w_2}^\sh (q) =  \sum_{w\in\{y_2,y_3,\ldots\}^*} c_w(w_1\ast w_2) G_w^\sh (q),\]
where we write $G^\sh_w(q)=G^\sh_{n_1,\ldots,n_r}(q)$ for each word $w=y_{n_1}\cdots y_{n_r}$.
\end{theorem}
\begin{proof} 
Consider the following holomorphic function on the upper half-plane: for integers $L,M>0$
\[ G^{(L,M)}_{n_1,\ldots,n_r}(\tau) = \sum_{\substack{0\prec \lambda_1\prec \cdots\prec \lambda_r\\ \lambda_i\in\Z_L \tau+\Z_M}} \frac{1}{\lambda_1^{n_1}\cdots \lambda_r^{n_r}}.\]
By definition, it follows that these functions satisfy the harmonic product: for any words $w_1,w_2\in\h^1$, one has
\begin{equation}\label{eq4_7}
G^{(L,M)}_{w_1}(\tau) G^{(L,M)}_{w_2}(\tau) = \sum_{w\in\{y_1,y_2,y_3,\ldots\}^*} c_w (w_1\ast w_2) G^{(L,M)}_w(\tau) .
\end{equation}
Since the space $\h^2:=\Q \langle y_2,y_3,y_4,\ldots\rangle$ is closed under the harmonic product $\ast$, taking $\lim_{L\rightarrow \infty}\lim_{M\rightarrow \infty}$ for both sides of \eqref{eq4_7}, one has for words $w_1,w_2\in\h^2$ 
\[ G_{w_1}(\tau) G_{w_2}(\tau) = \sum_{w\in\{y_2,y_3,y_4,\ldots\}^*} c_w (w_1\ast w_2) G_w(\tau).\]
Then the result follows from Theorem \ref{4_10}.
\end{proof}

We finally prove Theorem \ref{1_2}.
\begin{proof}[Proof of Theorem \ref{1_2}]
The precise statement of Theorem \ref{1_2} is as follows.
\begin{theorem}\label{4_12}
For any $w_1,w_2\in\h^2$, we obtain 
\[ \sum_{w\in\{y_1,y_2,y_3,\ldots\}^*} c_w (w_1\ \sh \ w_2) G_w^\sh(q) = \sum_{w\in\{y_2,y_3,\ldots\}^*} c_w (w_1\ast w_2) G_w^\sh(q),\]
which is called the restricted finite double shuffle relation in this paper.
\end{theorem}
\noindent
It follows from Propositions \ref{3_4}, \ref{3_7} and Theorem \ref{4_6} that $G^\sh$'s satisfies the shuffle product formula (I2).
With this, Theorem \ref{4_12} follows from Theorem \ref{4_11}.
\end{proof}

\begin{example}\label{4_13}
The first example of $\Q$-linear relations among $G^{\sh}$'s obtained in Theorem \ref{4_12} is 
\begin{equation*}
 G^{\sh}_4(q)-4 G^{\sh}_{1,3}(q)=0,
\end{equation*}
which comes from $y_2\ast y_2 - y_2\ \sh \ y_2$.
\end{example}

\begin{remark}\label{eq4_14}
Let us discuss the dimension of the space of $G^\sh$'s.
For our convenience, we take the normalization 
\[ \cG^\sh_{n_1,\ldots,n_r}(q) = \frac{1}{(-2\pi \sqrt{-1} )^{n_1+\cdots+n_r} } G^\sh_{n_1,\ldots,n_r}(q) .\]
As usual, we call $n_1+\cdots+n_r$ the weight and $\gamma_{n_1,\ldots,n_r}$ admissible if $n_r\ge2$.
Let $\mathcal{E}_N$ be the $\Q$-vector space spanned by all admissible $\cG^\sh$'s of weight $N$.
Set $\mathcal{E}_0=\Q$. 
The second author performed numerical experiments of the dimension of the above vector spaces up to $N=7$ by using Mathematica.
The list of the conjectural dimension is given as follows.
\begin{center}
$\begin{array}{c|ccccccccccccccccccccccccccc}
N &2&3&4&5&6&7\\\hline
\dim_{\Q} \mathcal{E}_N &1&2&3&6&10&18\\
\sharp\mbox{\ of admissible indices} &1&2&2^2&2^3&2^4&2^5
\end{array}$
\end{center}
Using Theorem \ref{4_12}, one can reduce the dimension of the space $\mathcal{E}_N$. 
The following is the table of the number of linearly independent relations provided by Theorem \ref{4_12}:
\begin{center}
$\begin{array}{c|ccccccccccccccccccccccccccc}
N & 0&1&2&3&4&5&6&7&8&9&10\\ \hline
\sharp \ {\rm of\ relations}& 0&0&0&0&1&1&3&5&11&19&37
\end{array}$
\end{center}
This table together with the dimension table shows that the relations obtained in Theorem \ref{4_12} are not enough to capture all relations among $\cG^\sh$'s for $N\ge5$.
\end{remark}

\begin{remark}
The above table of $\dim_\Q \mathcal{E}_N$ up to $N=7$ seemingly coincide with the table of $d_k'$ appeared in \cite[Table 5]{BK} which counts the dimension of the space of the generating series $[n_1,\ldots,n_r]$ of multiple divisor sums.
They suggest that the sequence $\{d_k'\}_{k\ge2}$ is given by $d_k'=2d_{k-2}'+2d_{k-3}'$ for $k\ge5$ with the initial values $d_2'=1,d_3'=2,d_4'=3$.
The $q$-series $\cG^\sh$ looks intimately related to the $q$-series $[n_1,\ldots,n_r]$, although we have no idea what the explicit relationships are.
\end{remark}

\begin{remark}
It is also worth mentioning that the $\Q$-algebra $\mathcal{E}$ contains the ring of quasimodular forms for ${\rm SL}_2(\Z)$ over $\Q$:
\[ \Q[\cG^{\sh}_2,\cG^\sh_4,\cG^\sh_6] \subset \mathcal{E},\]
and that the ring of quasimodular forms is stable under the derivative ${\rm d}=qd/dq$ (see \cite{KanZag}).
It might be remarkable to consider whether the $\Q$-algebra $\mathcal{E}$ is stable under the derivative, because by expressing ${\rm d}\cG^\sh$ as $\Q$-linear combinations of $\cG^\sh$'s and taking the constant term as an element in $\C[[q]]$ one obtains $\Q$-linear relations among multiple zeta values.
For instance, Kaneko \cite{Kaneko} proved the identity
\[ {\rm d}\cG^{\sh}_{N}(q) = 2N \big(\cG_{N+2}^{\sh} (q)-\sum_{i=1}^{N} \cG^{\sh}_{i,N+2-i}(q) \big),\] 
which provides the well-known formula $\zeta(N+2)=\sum_{i=1}^N \zeta(i,N+2-i)$.
We hope to discuss these problems in a future publication.
\end{remark}


\end{document}